\documentclass[11pt,letter]{article}
\usepackage[paperwidth=8.5in,left=1.0in,right=1.0in,top=1.0in,bottom=1.0in,
  paperheight=11.0in]{geometry}
\usepackage{cite}
\usepackage{color}
\usepackage{soul}
\usepackage{graphicx}
\usepackage{amsthm}
\usepackage{algorithm}
\usepackage[noend]{algorithmic}
\usepackage{hyperref}
\usepackage{chngcntr}
\usepackage{epstopdf}
\usepackage{amsmath}
\usepackage{amssymb}

\usepackage{caption}
\usepackage{subcaption}

\makeatletter
\def\@copyrightspace{\relax}
\makeatother

\newtheorem{definition}{Definition}

\newtheorem{theorem}{Theorem}
\newtheorem{lemma}{Lemma}
\newtheorem{proposition}{Proposition}
\newtheorem{corollary}{Corollary}

\newtheorem{observation}{Observation}

\newtheorem*{remark}{Remark}

\title{Doubly Balanced Connected Graph Partitioning}
\author{Saleh Soltan\thanks{Research partially supported by DTRA grant HDTRA1-13-1-0021 and CIAN NSF ERC under grant EEC-0812072.}\\
       {Electrical Engineering}\\
       {Columbia University}\\
       {New York, NY}\\
       {saleh@ee.columbia.edu}
\and Mihalis Yannakakis\thanks{Research partially supported by NSF grant CCF-1320654.}\\
       {Computer Science}\\
       {Columbia University}\\
       {New York, NY}\\
       {mihalis@cs.columbia.edu}
\and Gil Zussman$^\ast$\\
       {Electrical Engineering}\\
       {Columbia University}\\
       {New York, NY}\\
       {gil@ee.columbia.edu}
}
\date{}
\begin{document}
\setlength{\textfloatsep}{10 pt}
\clearpage\maketitle
\thispagestyle{empty}
\begin{abstract}
We introduce and study the Doubly Balanced Connected graph Partitioning (DBCP) problem: Let $G=(V,E)$ be a connected graph with a weight (supply/demand) function $p:V\rightarrow \{-1,+1\}$ satisfying $p(V)=\sum_{j\in V} p(j)=0$.
The objective is to partition $G$ into $(V_1,V_2)$ such that $G[V_1]$ and $G[V_2]$ are connected,  $|p(V_1)|,|p(V_2)|\leq c_p$, and $\max\{\frac{|V_1|}{|V_2|},\frac{|V_2|}{|V_1|}\}\leq c_s$, for some constants $c_p$ and $c_s$. When $G$ is 2-connected, we show that a solution with $c_p=1$ and  $c_s=3$ always exists and can be found in polynomial time. Moreover, when $G$ is 3-connected, we show that there is always a `perfect' solution (a partition with $p(V_1)=p(V_2)=0$ and $|V_1|=|V_2|$, if $|V|\equiv 0 (\mathrm{mod}~4)$), and it can be found in polynomial time. Our techniques can be extended, with similar results, to the case in which the weights are arbitrary (not necessarily $\pm 1$), and  to the case that $p(V)\neq 0$ and the excess supply/demand should be split evenly.
They also apply to the problem of partitioning a graph with two types of nodes into two large connected subgraphs that preserve approximately the proportion of the
two types.
\end{abstract}
\newpage
\section{Introduction}
Power Grid Islanding is an effective method to mitigate cascading failures in power grids~\cite{sun2002two}. The challenge is to partition the network into smaller connected components, called \emph{islands},  such that each island can operate independently for a while. In order for an island to operate, it is necessary that the power supply and demand at that island are almost equal.\footnote{If the supply and demand are not exactly equal but still relatively close, load shedding/generation curtailing can be used in order for the island to operate.} Equality of supply and demand in an island, however, may not be sufficient for its independent operation. It is also important that the infrastructure in that island has the physical capacity to safely transfer the power from the supply nodes to the demand nodes. When the island is large enough compared to the initial network, it is more likely that it has enough capacity. This problem has been studied in the power systems community but almost all the algorithms provided in the literature are heuristic methods that have been shown to be effective only by simulations~\cite{sun2002two,sanchez2014hierarchical,pahwa2013optimal,fan2012mixed}.


Motivated by this application, we formally introduce and study the Doubly Balanced Connected graph Partitioning (DBCP) problem:
Let $G=(V,E)$ be a connected graph with a weight (supply/demand) function $p:V\rightarrow \mathbb{Z}$ satisfying $p(V)=\sum_{j\in V} p(j)=0$. The objective is to partition $V$ into $(V_1,V_2)$ such that $G[V_1]$ and $G[V_2]$ are connected, $|p(V_1)|,|p(V_2)|\leq c_p$, and $\max\{\frac{|V_1|}{|V_2|},\frac{|V_2|}{|V_1|}\}\leq c_s$, for some constants $c_p$ and $c_s$. We also consider the case that  $p(V)\neq 0$, in which the
excess supply/demand should be split roughly evenly.

The problem calls for a partition into two connected subgraphs that simultaneously balances two objectives, (1) the supply/demand within each part, and (2) the sizes of the parts.
The connected partitioning problem with only the size objective has been studied previously.
In the most well-known result, Lov\'{a}z and Gyori~\cite{lovasz1977homology,gyori1976division} independently proved, using different methods, that every $k$-connected graph can be partitioned into $k$ arbitrarily sized connected subgraphs. However, neither of the proofs is constructive, and there are no known
polynomial-time algorithms to find such a partition for $k>3$.
For $k=2$, a linear time algorithm is provided in~\cite{suzuki1990linear} and for $k=3$
an $O(|V|^2)$ algorithm is provided in~\cite{wada1994efficient}.\footnote{For $k=2$, a much simpler approach than the one in~\cite{suzuki1990linear} is to use the $st$-numbering~\cite{lempel1967algorithm} for 2-connected graphs.}
The complexity of the problem with the size objective and related optimization problems
have been studied in~\cite{dyer,chlebikova1996approximating,chataigner2007approximation}
and there are various NP-hardness and inapproximability results.
Note that the size of the cut is not of any relevance here (so the
extensive literature on finding balanced partitions, not necessarily connected,
that minimize the cut is not relevant.)

The objective of balancing the supply/demand alone, when all $p(i)$ are $\pm1$, can also be seen as an extension for the objective of balancing the size (which corresponds to $p(i)=1$).
Our bi-objective problem of balancing both supply/demand and size, can be seen also
as an extension of the problem of finding a partition that balances the size for two types of nodes
simultaneously: Suppose the nodes of a graph are partitioned into red and blue nodes.
Find a partition of the graph into two large connected subgraphs that splits approximately evenly both
the red and the blue nodes.

We now summarize our results and techniques. Since the power grids are designed to withstand a single failure (``$N-1$" standard)~\cite{bienstock2016electrical}, and therefore 2-connected, our focus is mainly on the graphs that are at least 2-connected.
We first, in Section \ref{sec:BPGI}, study  the connected partitioning problem with
only the supply/demand balancing objective, and show results that
parallel the results for balancing size alone, using similar techniques:
The problem is NP-hard in general.
For 2-connected graphs and weights $p(i)=\pm1$, there is always a perfectly balanced partition
and we can find it easily using an $st$-numbering. For 3-connected graphs and weights $p(i)=\pm1$,
there is a perfectly balanced partition into three connected graphs, and we can find it
using a nonseparating ear decomposition of 3-connected graphs~\cite{cheriyan1988finding} and similar ideas as in~\cite{wada1994efficient}.

The problem is more challenging when we deal with both balancing objectives, supply/demand and size.
This is the main focus and occupies the bulk of this paper.
Our main results are existence results and algorithms for 2- and 3-connected graphs.
It is easy to observe that we cannot achieve perfection in one objective ($c_p=0$ or $c_s =1$)
without sacrificing completely the other objective.
We show that allowing the supply/demand of the parts to be off balance by at most the weight of one node
suffices to get a partition that is roughly balanced also with respect to size.

First, in Section~\ref{subsec:3-connected}, we study the case of 3-connected graphs
since we use this later as the basis of handling 2-connected graphs.
We show that if $\forall i,~p(i)=\pm1$, there is a partition that is perfectly balanced
with respect to both objectives, if $|V|\equiv0 (\mathrm{mod}~4)$ (otherwise the
sizes are slightly off for parity reasons); for general $p$, the partition is perfect
in both objectives up to the weight of a single node. Furthermore, the partition can be constructed in
polynomial time. Our approach uses the convex embedding characterization of $k$-connectivity
studied by Linial, Lov\'{a}z, and Wigderson~\cite{linial1988rubber}. We need to adapt it for
our purposes so that the convex embedding also has certain desired geometric properties,
and for this purpose we use the nonseparating ear decomposition of 3-connected graphs
of~\cite{cheriyan1988finding} to obtain a suitable embedding.

Then, in Section~\ref{subsec:2-connected}, we analyze the case of 2-connected graphs.
We reduce it to two subcases: either (1) there is a separation pair that splits the graph
into components that are not very large, or (2) we can perform a series of
contractions to achieve a 3-connected graph whose edges represent
contracted subgraphs that are not too large. We provide a good partitioning algorithm
for case (1), and for case (2) we extend the algorithms for 3-connected graphs
to handle also the complications arising from edges representing contracted subgraphs.
Finally, in Section~\ref{sec:blue_red}, we briefly discuss the  problem
of finding a connected partitioning of a graph with two types of nodes
that splits roughly evenly both types.

The body of the paper contains proofs and sketches for some of the results,
and Appendices~\ref{sec:proof1}, \ref{sec:proof2}, and \ref{sec:proof3} contain the remaining proofs.
Graph-theoretic background and definitions (e.g., the notions of $st$-numbering,
nonseparating ear decomposition, convex embedding of $k$-connected graphs)
are given in Appendix~\ref{sec:pre}.

\section{Balancing the Supply/Demand Only}\label{sec:BCPI}
In this section, we study the single objective problem of
finding a partition of the graph into connected subgraphs that balances (approximately)
the supply and demand in each part of the partition, without any regard
to the sizes of the parts.
We can state the optimization problem as follows, and will refer to it as the Balanced Connected Partition with Integer weights (BCPI) problem.
\begin{definition}
Given a graph $G=(V,E)$ with a weight (supply/demand) function $p~:~V~\rightarrow~\mathbb{Z}$ satisfying $\sum_{j\in V} p(j)=0$. The BCPI problem is the problem of partitioning $V$ into $(V_1,V_2)$ such that
\begin{enumerate}
\item $V_1\cap V_2=\emptyset$ and $V_1\cup V_2=V$,
\item $G[V_1]$ and $G[V_2]$ are connected,
\item $|p(V_1)|+|p(V_2)|$ is minimized, where $p(V_i)=\sum_{j\in V_i} p(j).$
\end{enumerate}
\end{definition}

Clearly, the minimum possible value for $|p(V_1)|+|p(V_2)|$ that we can hope for is 0,
which occurs iff $p(V_1)=p(V_2) =0$.
It is easy to show that the problem of determining whether there exists such a `perfect' partition
(and hence the BCPI problem) is strongly NP-hard.
The proof is very similar to analogous results concerning the partition of a graph
into two connected subgraphs with equal sizes (or weights, when nodes have positive weights)
\cite{chataigner2007approximation,dyer}

\begin{proposition}\label{lem:BCPI_hard}
(1) It is strongly NP-hard to determine whether there is a solution to the BCPI problem with value 0, even when $G$ is 2-connected. \\
(2) If $G$ is not 2-connected, then this problem is NP-hard even when $\forall i, p(i)=\pm1$.
\end{proposition}

Although it is NP-hard to tell whether there is a solution satisfying $p(V_1)=p(V_2) =0$,
even when $\forall i, p(i)=\pm1$, in this case, if the graph $G$ is 2-connected there is always such a solution.
For general weights $p$, there is a solution such that $|p(V_1)|,|p(V_2)|\leq \max_{j\in V}|p(j)|/2$
and it can be found easily in linear time using the $st$-numbering between two nodes (see the Appendix \ref{sec:proof1}).
\begin{proposition}\label{lem:2_poly}
Let $G$ be a 2-connected graph and $u,v$ any two nodes in $V$ such that $p(u) p(v)>0$.\\
(1) There is a solution such that $u\in V_1$, $v\in V_2$, and $|p(V_1)|=|p(V_2)|\leq \max_{j\in V}|p(j)|/2$.\\
(2) If $\forall i, p(i)=\pm1$, we can find a solution such that $u\in V_1$, $v\in V_2$, and $p(V_1)=p(V_2)=0$.\\
In both cases, the solution can be found in $O(|E|)$ time.
\end{proposition}
\begin{remark}
The bound in Proposition~\ref{lem:2_poly} (1) is tight. A simple example is a cycle of length 4 like $(v_1,v_2,v_3,v_4)$ with $p(v_1)=-p$, $p(v_2)=-p/2$, $p(v_3)=p$, and $p(v_4)=p/2$. It is easy to see that in this example $|p(V_1)|+|p(V_2)|= \max_{j\in V}|p(j)|=p$ is the best that one can do.
\end{remark}





\subsection{Connected Partitioning into Many Parts}
The BCPI problem can be extended to partitioning a graph into $k=3$ or more parts.
Let $G=(V,E)$ be a graph with a weight function $p~:~V~\rightarrow~\mathbb{Z}$ satisfying $\sum_{j\in V} p(j)=0$. The $\text{BCPI}_k$ problem is the problem of partitioning $G$ into $(V_1,V_2,\dots,V_k)$ such that for any $1\leq i\leq k$, $G[V_i]$ is connected and $\sum_{i=1}^k |p(V_i)|$ is minimized.

In the following proposition, we show that for $k=3$, if $ p(i)=\pm1, \forall i$, then
there is always a perfect partition (i.e., with $p(V_1)=p(V_2)=p(V_3)=0$) and it can be found efficiently.
For general $p$, we can find a partition such that $|p(V_1)|+|p(V_2)|+|p(V_3)|\leq 2 \max_{j\in V}|p(j)|$.
The proof and the algorithm use a similar approach as the algorithm in~\cite{wada1994efficient}
for partitioning a 3-connected graph to three connected parts with prescribed sizes, using the nonseparating ear decomposition of 3-connected graphs as described in Subsection~\ref{subsec:ear-decomposition}.
The proof is given in the Appendix \ref{sec:proof1}.

\begin{proposition}\label{lem:3-connected-zero-p}
Let $G$ be a 3-connected graph and $u,v,w$ three nodes in $V$ such that $p(u), p(v),p(w)>0$ or $p(u), p(v),p(w)<0$.\\
(1) There is a solution such that $u\in V_1$, $v\in V_2$, $w\in V_3$, and $|p(V_1)|+|p(V_2)|+|p(V_3)|\leq 2 \max_{j\in V}|p(j)|$. \\
(2) If $\forall i, p(i)=\pm1$, then there is a solution such that $u\in V_1$, $v\in V_2$, $w\in V_3$, and
$|p(V_1)|=|p(V_2)|=|p(V_3)|=0$.\\
In both cases, the solution can be found in $O(|E|)$ time.
\end{proposition}


\section{Balancing Both Objectives}\label{sec:BPGI}
In this section, we formally define and study the Doubly Balanced Connected graph Partitioning (DBCP) problem.

\begin{definition}
Given a graph $G=(V,E)$ with a weight (supply/demand) function $p~:~V~\rightarrow~\mathbb{Z}$ satisfying $p(V)=\sum_{j\in V} p(j)=0$ and constants $c_p\geq 0$, $c_s\geq 1$. The \emph{DBCP problem} is the problem of partitioning $V$ into $(V_1,V_2)$ such that
\begin{enumerate}
\item $V_1\cap V_2=\emptyset$ and $V_1\cup V_2=V$,
\item $G[V_1]$ and $G[V_2]$ are connected,
\item $|p(V_1)|,|p(V_2)|\leq c_p$ and $\max\{\frac{|V_1|}{|V_2|},\frac{|V_2|}{|V_1|}\}\leq c_s$, where $p(V_i)=\sum_{j\in V_i} p(j).$
\end{enumerate}
\end{definition}

\begin{remark}
Our techniques apply also to the case that $p(V)\neq0$.
In this case, the requirement 3 on $p(V_1)$ and $p(V_2)$ is
$|p(V_1)-p(V)/2|,|p(V_2)-p(V)/2|\leq c_p$, i.e., the excess supply/demand
is split approximately evenly between the two parts.
\end{remark}

We will concentrate on 2-connected and 3-connected graphs and
show that we can construct efficiently good partitions.
For most of the section we will focus on the case that $p(i)=\pm1, \forall i \in V$.
This case contains all the essential ideas.
All the techniques generalize to the case of arbitrary $p$,
and we will state the corresponding theorems.

We observed in Section 2 that if the graph is 2-connected and $p(i)=\pm1, \forall i \in V$ then there is always a connected partition
that is perfect with respect to the weight objective,
$p(V_1)=p(V_2)=0$, i.e., (3) is satisfied with $c_p=0$.
We know also from \cite{lovasz1977homology,gyori1976division}
that there is always a connected partition
that is perfect with respect to the size objective,
$|V_1|=|V_2|$, i.e., condition 3 is satisfied with $c_s=1$.
The following observations show that combining the two objectives makes
the problem more challenging.
If we insist on $c_p=0$, then $c_s$ cannot be bounded in general,
(it will be $\Omega(|V|)$),
and if we insist on $c_s=1$, then $c_p$ cannot be bounded.
The series-parallel graphs of Figure~\ref{fig:example_bound} provide simple counterexamples.

\begin{figure}[t]
\centering
\includegraphics[scale=0.6]{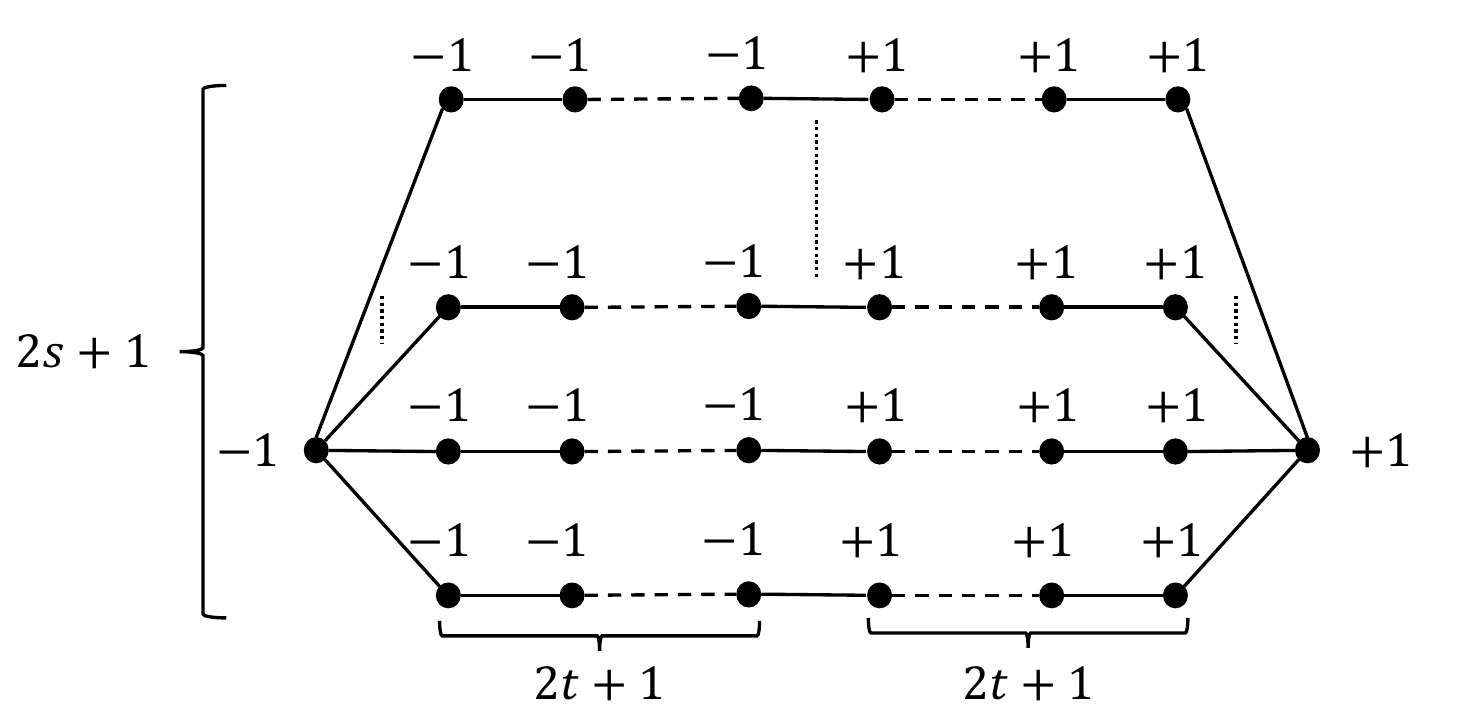}
\caption{Series-parallel graphs with $2s+1$ paths of length $4t+2$ used in Observations~\ref{obs:p=0} and \ref{obs:s=1}.}
\label{fig:example_bound}
\end{figure}
\begin{observation}\label{obs:p=0}
If $c_p=0$, then for any $c_s<|V|/2-1$, there exist a 2-connected graph $G$ such that the DBCP problem does not have a solution even when $\forall i, p(i)=\pm1$.
\end{observation}
\begin{proof}
In the graph depicted in Figure~\ref{fig:example_bound}, set $t=0$.
\end{proof}
\begin{observation}\label{obs:s=1}
If $c_s=1$, then for any $c_p<|V|/6$, there exist a 2-connected graph $G$ such that the DBCP problem does not have a solution even when $\forall i, p(i)=\pm1$.
\end{observation}
\begin{proof}
In the graph depicted in Figure~\ref{fig:example_bound}, set $s=1$.
\end{proof}

Thus, $c_p$ has to be at least 1 to have any hope for a bounded $c_s$.
We show in this section that $c_p=1$ suffices for all 2-connected graphs.
We first treat 3-connected graphs.

\subsection{3-Connected Graphs}\label{subsec:3-connected}
Let $G=(V,E)$ be a 3-connected graph.
Assume for the most of this section that $\forall i, p(i)=\pm1$ and $p(V)=0$
(we will state the results for general $p$ at the end).
We show that $G$ has a partition that is essentially perfect with
respect to both objectives, i.e., with $c_p=0$ and $c_s=1$.
We say ``essentially", because $p(V_1)=p(V_2)=0$ and $|V_1|=|V_2|$
imply that $|V_1|=|V_2|$ are even, and hence $V$ must be a multiple of 4.
If this is the case, then indeed we can find such a perfect partition.
If $|V|\equiv 2 (\mathrm{mod}~4)$ ($|V|$ has to be even since $p(V)=0$), then
we can find an `almost perfect' partition, one in which
$|p(V_1)|=|p(V_2)|=1$ and $|V_1|=|V_2|$ (or one in which
$p(V_1)=p(V_2)=0$ and $|V_1|=|V_2|+2$).

We first treat the case that $G$ contains a triangle (i.e., cycle of length 3).
In the following Lemma, we use the embedding for $k$-connected graphs introduced in~\cite{linial1988rubber} and as described in Subsection~\ref{subsec:embedding}, to show that if $G$ is 3-connected with a triangle and all weights are $\pm 1$, then the DBCP problem has a perfect solution.

\begin{lemma}\label{lem:3-connected-tri}
If $G$ is 3-connected with a triangle, $\forall i, p(i)=\pm1$, and $|V|\equiv0(\mathrm{mod}~4)$, then there exists a solution to the DBCP problem
with $p(V_1)=p(V_2)=0$ and $|V_1|=|V_2|$.
If $|V|\equiv2 (\mathrm{mod}~4)$, then there is a solution with $p(V_1)=p(V_2)=0$ and $|V_1|=|V_2|+2$.
Moreover, this partition can be found in polynomial time.
\end{lemma}
\begin{proof}
Assume that $|V|\equiv0(\mathrm{mod}~4)$; the proof for the case $|V|\equiv2(\mathrm{mod}~4)$ is similar.
In~\cite{linial1988rubber} as described in Subsection~\ref{subsec:embedding}, it is proved that if $G$ is a $k$-connected graph, then for every $X\subset V$ with $|X|=k$, $G$ has a convex $X$-embedding in general position. Moreover, this embedding can be found by solving a set of linear equations of size $|V|$.
Now, assume $v,u,w\in V$ form a triangle in $G$. Set $X=\{v,u,w\}$. Using the theorem, $G$ has a convex $X$-embedding $f:V\rightarrow\mathbb{R}^2$ in general position. Consider a circle $\mathcal{C}$ around the triangle $f(u),f(v),f(w)$ in $\mathbb{R}^2$ as shown in an example in Fig.~\ref{fig:embedding_1}. Also consider a directed line $\mathcal{L}$ tangent to the circle $C$ at point $A$. If we project the nodes of $G$ onto the line $\mathcal{L}$, since the embedding is convex and also $\{u,v\},\{u,w\},\{w,v\}\in E$, the order of the nodes' projection gives an $st$-numbering between the first and the last node (notice that the first and last nodes are always from the set $X$). For instance in Fig.~\ref{fig:embedding_1}, the order of projections give an $st$-numbering between the nodes $u$ and $v$ in $G$. Hence, if we set $V_1$ to be the $|V|/2$ nodes whose projections come first and $V_2$ are the $|V|/2$ nodes whose projections come last, then $G[V_1]$ and $G[V_2]$ are both connected and $|V_1|=|V_2|=|V|/2$. The only thing that may not match is $p(V_1)$ and $p(V_2)$. Notice that for each directed line tangent to the circle $\mathcal{C}$, we can similarly get a partition such that $|V_1|=|V_2|=|V|/2$. So all we need is a point $D$ on the circle $\mathcal{C}$ such that if we partition based on the directed line tangent to $C$ at point $D$, then $p(V_1)=p(V_2)=0$. To find such a point, we move $\mathcal{L}$ from being tangent at point $A$ to point $B$ ($AB$ is a diameter of the circle $\mathcal{C}$) and consider the resulting partition. Notice that if at point $A$, $p(V_1)>0$, then at point $B$ since $V_1$ and $V_2$ completely switch places compared to the partition at point $A$,  $p(V_1)<0$. Hence, as we move $\mathcal{L}$ from being tangent at point $A$ to point $B$ and keep it tangent to the circle, in the resulting partitions, $p(V_1)$ goes from some positive value to a non-positive value. Notice that the partition $(V_1,V_2)$ changes only if $\mathcal{L}$ passes a point $D$ on the circle such that at $D$, $\mathcal{L}$ is perpendicular to a line that connects $f(i)$ to $f(j)$ for some $i,j\in V$. Now, since the embedding is in general position, there are exactly two points on every line that connects two points $f(i)$ and $f(j)$, so $V_1$ changes at most by one node leaving $V_1$ and one node entering  $V_1$ at each step as we move $\mathcal{L}$. Hence, $p(V_1)$ changes by either $\pm 2$ or $0$ value at each change. Now, since $|V|\equiv 0 (\mathrm{mod}~ 4)$, $p(V_1)$ has an even value in all the resulting partitions. Therefore, as we move $\mathcal{L}$ from being tangent at point $A$ to point $B$, there must be a point $D$ such that in the resulted partition $p(V_1)=p(V_2)=0$.

It is also easy to see that  since $V_1$ may change only when a line that passes through 2 nodes of graph $G$ is perpendicular to $\mathcal{L}$, we can find $D$ in at most $O(|V|^2)$ steps.
\end{proof}
\begin{figure}[t]
\centering
\includegraphics[scale=0.75]{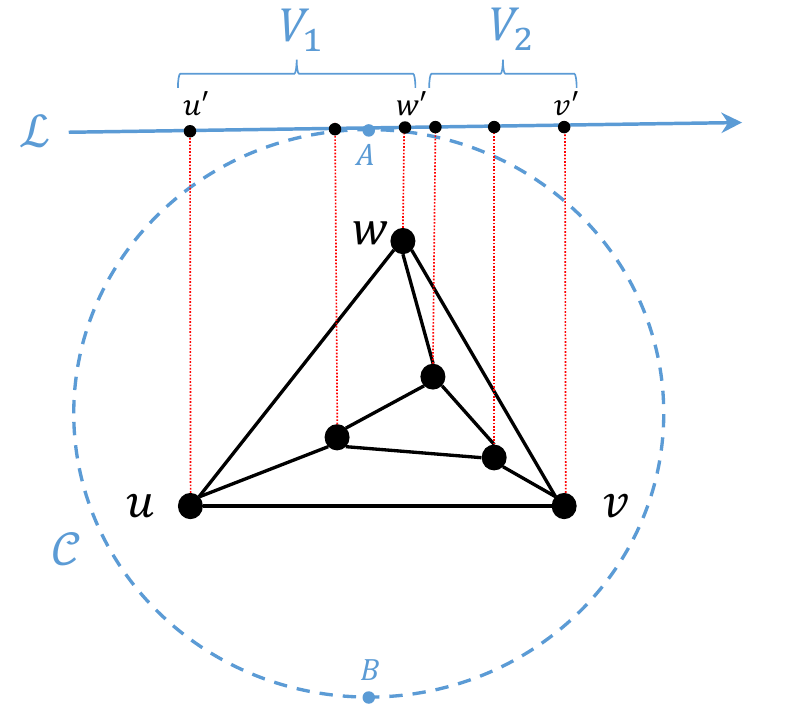}
\caption{Proof of Lemma~\ref{lem:3-connected-tri}.}
\label{fig:embedding_1}
\end{figure}
When $G$ is a triangle-free 3-connected graph, however, the proof in Lemma~\ref{lem:3-connected-tri} cannot be directly used anymore. The reason is if for example $\{u,v\}\notin E$ and we project the nodes of $G$ onto the line $\mathcal{L}$, this time the order of the nodes projection does not give an $st$-numbering between the first and the last node if  for example $u$ and $w$ are the first and last node, since some of the middle nodes may only be connected to $v$. To prove a similar result for triangle-free 3-connected case, we first provide the following two Lemmas. The main purpose of the following two Lemmas are to compensate for the triangle-freeness of $G$ in the proof of Lemma~\ref{lem:3-connected-tri}. The idea is to show that in every 3-connected graph, there is a triple $\{u,w,v\}\in V$, such that $\{u,w\},\{w,v\}\in E$ and in every partition that we get by the approach used in the proof of Lemma~\ref{lem:3-connected-tri}, if $u$ and $v$ are in $V_i$, so is a path between $u$ and $v$.

\begin{lemma}\label{lem:cut}
If $G$ is 3-connected, then there exists a set $\{u,v,w\}\in V$ and a partition of $V$ into $(V_1',V_2')$ such that:
\begin{enumerate}
\item $V_1'\cap V_2' =\emptyset$ and $V_1'\cup V_2'=V$,
\item $G[V_1']$ and  $G[V_2']$ are connected,
\item $\{u,w\},\{v,w\}\in E$,
\item $w\in V_1'$, $u,v\in V_2'$,
\item $|V_2'|\leq |V|/2$.
\end{enumerate}
Moreover, such a partition and $\{u,v,w\}$ can be found in $O(|E|)$ time.
\end{lemma}
\begin{proof}
Using the algorithm presented in~\cite{cheriyan1988finding}, we can find a non-separating cycle $C_0$ in $G$ such that every node in $C_0$ has a neighbor in $G\backslash C_0$. Now, we consider two cases:
\begin{itemize}
\item[(i)] If $|C_0|\leq |V|/2+1$, then select any three consecutive nodes $(u,w,v)$ of $C_0$ and set $V_2'=C_0\backslash \{w\}$ and $V_1'=V\backslash V_2'$.
\item[(ii)] If $|C_0|> |V|/2+1$, since every node in $C_0$ has a neighbor in $G\backslash C_0$, there exists a node $w\in V\backslash C_0$ such that $|N(w)\cap C_0|\geq 2$. Select two nodes $u,v\in N(w)\cap C_0$. There exists a path $P$ in $C_0$ between $u$ and $v$ such that $|P|<|V|/2-1$. Set $V_2'=\{u,v\}\cup P$ and $V_1'=V\backslash V_2'$.\qedhere
\end{itemize}
\end{proof}
\begin{figure}[t]
\centering
\includegraphics[scale=0.9]{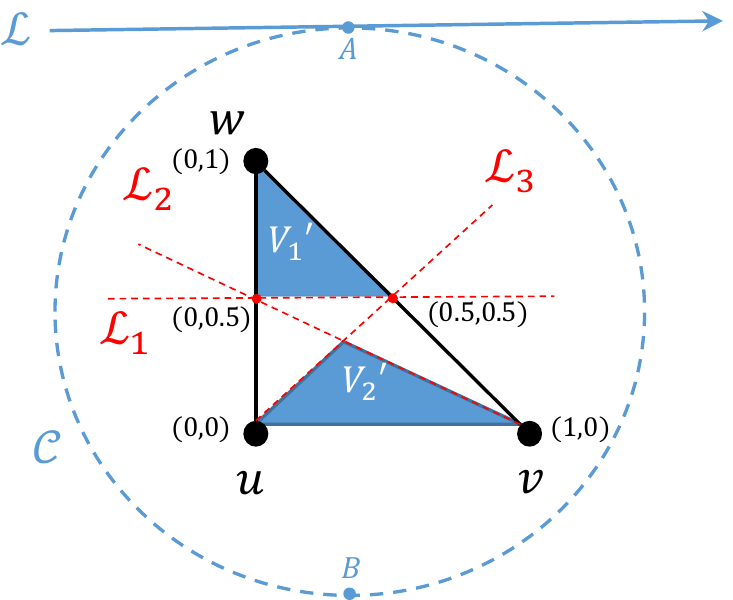}
\caption{Proof of Lemma~\ref{lem:embedding} and Theorem~\ref{th:3-connected}.}
\label{fig:3-connected}
\end{figure}
\begin{lemma}\label{lem:embedding}
Given a partition $(V_1',V_2')$ of a 3-connected graph $G$ with properties described in Lemma~\ref{lem:cut}, $G$ has a convex $X$-embedding in general position with mapping $f:V\rightarrow \mathbb{R}^2$ such that:
\begin{enumerate}
\item $X=\{u,w,v\}$, $f(u)=(0,0)$, $f(v)=(1,0)$, and $f(w)=(0,1)$,
\item Every node $i$ in $V_1'$ is mapped to a point $(f_1(i),f_2(i))$ with $f_2(i)\geq 1/2$,
\item Every node $i$ in $V_2'$ is mapped to a point $(f_1(i),f_2(i))$ with $f_1(i)\geq f_2(i)$ and $f_1(i)+2f_2(i)\leq 1$.
\end{enumerate}
Moreover, such an embedding can be found in in polynomial time.
\end{lemma}
\begin{proof}[Sketch of the proof]
Set $X=\{v,u,w\}$. Using~\cite{linial1988rubber}, $G$ has a convex $X$-embedding in $\mathbb{R}^2$ in general position with mapping $f:V\rightarrow \mathbb{R}^2$ such that $f(u)=(0,0)$, $f(v)=(1,0)$, and $f(w)=(0,1)$.  In the $X$-embedding of the nodes, we have a freedom to set the elasticity coefficient vector $\vec{c}$ to anything that we want (except a measure zero set of vectors). So for any edge $\{i,j\}\in G[V_1]\cup G[V_2]$, set $c_{ij}=g$; and for any $\{i,j\} \in E[V_1',V_2']$, set $c_{ij}=1$. Since both $G[V_1']$ and $G[V_2']$ are connected, as we increase $g$, nodes in $V_1'$ get closer to $w$ and nodes in $V_2'$ get closer to the line $uv$ (as $g\to \infty$, nodes in $V_1'$ get in the same position as $w$ and node in $V_2'$ get on the line $uv$). Hence, intuitively there exists a value $g$, for which all the nodes in $V_1$ are above line $\mathcal{L}_1$ and all the nodes in $V_2'$ are below the lines $\mathcal{L}_2$ and $\mathcal{L}_3$ as depicted in Fig.~\ref{fig:3-connected}. In the appendix we give the detailed proof and show that a $g$ with polynomially many bits suffices.
\end{proof}

Using Lemmas~\ref{lem:cut} and \ref{lem:embedding},  we are now able to prove that for any 3-connected graph $G$ such that all the weights are $\pm 1$, the DBCP problem has a solution for $c_p=0$ and $c_s=1$. The idea of the proof is similar to the proof of Lemma~\ref{lem:3-connected-tri}, however, we need to use Lemma~\ref{lem:cut} to find a desirable partition $(V_1',V_2')$ and then use this partition to find an embedding with properties as described in Lemma~\ref{lem:embedding}. By using this embedding, we
can show that in every partition that we obtain by the approach in the proof of Lemma~\ref{lem:3-connected-tri}, if $u$ and $v$ are in $V_i$, so is a path between $u$ and $v$. This implies then that $G[V_1]$ and $G[V_2]$ are connected. So we can use similar arguments as in the proof of Lemma~\ref{lem:3-connected-tri} to prove the following theorem (see the Appendix~\ref{sec:proof2} for the proof details).
\begin{theorem}\label{th:3-connected}
If $G$ is 3-connected,  $\forall i, p(i)=\pm1$, and $|V|\equiv0(\mathrm{mod}~4)$,
then there exists a solution to the DBCP problem
with $p(V_1)=p(V_2)=0$ and $|V_1|=|V_2|$.
If $|V|\equiv2(\mathrm{mod}~4)$, then there is a solution with $p(V_1)=p(V_2)=0$ and $|V_1|=|V_2|+2$.
Moreover, this partition can be found in polynomial time.
\end{theorem}

It is easy to check for a 3-connected graph $G$, by using the same approach as in the proof of Lemma~\ref{lem:3-connected-tri} and Theorem~\ref{th:3-connected}, that even when the weights are arbitrary (not necessarily $\pm 1$) and also $p(V)\neq 0$, we can still find a connected partition $(V_1,V_2)$ for $G$ such that
$|p(V_1)-p(V)/2|,|p(V_1)-p(V)/2| \leq \max_{i\in V}|p(i)|$ and $|V_1|=|V_2|$.
\begin{corollary}\label{cor:3-general}
If $G$ is 3-connected, then the DBCP problem (with arbitrary $p$, and not necessarily satisfying $p(V)= 0$) has a
solution for $c_p=\max_{i\in V}|p(i)|$ and $c_s=1$. Moreover, this solution can be found in polynomial time.
\end{corollary}

\subsection{2-Connected Graphs}\label{subsec:2-connected}
We first define a \emph{pseudo-path} between two nodes in a graph as below. The definition is inspired by the definition of the $st$-numbering.
\begin{definition}
A \emph{pseudo-path} between nodes $u$ and $v$ in $G=(V,E)$, is a sequence of nodes $v_1,\dots,v_t$ such that if $v_0= u$ and $v_{t+1}=v$, then
for any $1\leq i\leq t$, $v_i$ has neighbors $v_j$ and $v_k$ such that $j<i<k$.
\end{definition}
Using the pseudo-path notion, in the following lemma we show that if $G$ is 2-connected and has a separation pair such that none of the resulting components are too large, then the DBCP problem always has a solution for some $c_p=c_s=O(1)$. The idea used in the proof of this lemma is one of the building blocks of the proof for the general 2-connected graph case.

\begin{lemma}\label{lem:parallel}
Given a 2-connected graph $G$ and an integer $q>1$, if $\forall i: p(i)=\pm1$ and $G$ has a separation pair $\{u,v\}\subset V$ such that for every connected component $G_i=(V_i,E_i)$ of $G[V\backslash \{u,v\}]$,  $|V_i|< (q-1)|V|/q$, then the DBCP problem has a solution for $c_p=1$, $c_s=q-1$, and it can be found in $O(|E|)$ time.
\end{lemma}
\begin{proof}
Assume for simplicity that $|V|$ is divisible by $q$. There is a separation pair $\{u,v\}\in V$ such that if $G_1,\dots,G_k$ are the connected components of $G\backslash \{u,v\}$, for any $i$, $|V_i|< (q-1)|V|/q$.  Since $G$ is 2-connected, $G_1,\dots,G_k$ can be presented by pseudo-paths $P_1,\dots,P_k$ between $u$ and $v$. Assume $P_1,\dots,P_k$ are in increasing order based on their lengths. There exists two subsets of the pseudo-paths $S_1$ and $S_2$ such that $S_1\cap S_2=\emptyset$, $S_1\cup S_2=\{P_1,\dots,P_{k}\}$ and $\sum_{P_j\in S_i} |P_j| \geq |V|/q-1$ for $i=1,2$. The proof is very simple. If $|P_k|< |V|/q$, the greedy algorithm for the partition problem gives the desired partition of $\{P_1,\dots,P_{k}\}$. If $|P_k|\geq |V|/q$, since $|P_k|< (q-1)|V|/q$, $S_1=\{P_1,\dots,P_{k-1}\}$ and $S_2=\{P_k\}$ is the desired partition.

     Now, if we put all the pseudo-paths in $S_1$ back to back, they will form a longer pseudo-path $Q_1$ between $u$ and $v$. Similarly, we can form another pseudo-path $Q_2$ from the paths in $S_2$.  Without loss of generality we can assume $|Q_1|\geq |Q_2|$. From $u$, including $u$ itself, we count $|V|/q$ of the nodes in $Q_1$ towards $v$ and put them in a set $V'$. Without loss of generality, we can assume $p(V')\geq 0$. If $p(V')=0$, then $(V',V\backslash V')$ is a good partition and we are done. Hence, assume  $p(V')>0$. We keep $V'$ fixed and make a new set $V''$ by continuing to add nodes from $Q_1$ to $V'$ one by one until we get to $v$.
If $p(V'')$ hits 0 as we add nodes one by one, we stop and $(V'',V\backslash V'')$ is a good partition and we are done. So, assume $V''=Q_1\cup \{u\}$ and $p(V'')>0$. Since $|Q_2\cup\{v\}|\geq |V|/q$, $|V''|\leq (q-1)|V|/q$. If $|V''| < (q-1)|V|/q$, we add nodes from $Q_2$ one by one toward $u$ until either $|V''|=0$ or $|V''| = (q-1)|V|/q$. If we hit 0 first (i.e., $p(V'')=0$) and $|V''|\leq (q-1)|V|/q$, define $V_1=V''\backslash\{u\}$, then $(V_1,V\backslash V_1)$ is a good partition. So assume $|V''|=(q-1)|V|/q$ and $p(V'')>0$. Define $V'''=V\backslash V''$. Since $p(V'')>0$ and $|V''|=(q-1)|V|/q$, then $p(V''')<0$ and $|V'''|=|V|/q$. Also notice that $V'''\subseteq Q_2$. We consider two cases. Either $|p(V')|\geq|p(V''')|$ or $|p(V')|<|p(V''')|$.

        If $|p(V')|\geq|p(V''')|$, then if we start from $u$ and pick nodes one by one from $Q_1$ in order, we can get a subset $V'_1$ of $V'$  such that $|p(V'_1)|=|p(V''')|$. Hence, if we define $V_1=V_1'\cup V'''$, then $(V_1,V\backslash V_1)$ is a good partition.

        If $|p(V')|<|p(V''')|$, then we can build a new set $V_1$ by adding nodes one by one from $Q_2$ to $V'$ until $P(V_1)=0$. It is easy to see that since $|p(V')|<|p(V''')|$, then $V_1\backslash V'\subseteq V'''$. Hence, $(V_1, V\backslash V_1)$ is a good partition.
\end{proof}

\begin{corollary}\label{cor:serpar}
If $G$ is a 2-connected series-parallel graph and $\forall i: p(i)=\pm1$, then the DBCP problem has a solution for $c_p=1,~c_s=2$, and the solution can be found in $O(|E|)$ time.
\end{corollary}
The graph in Figure~\ref{fig:example_bound} with $s=1$ shows that these parameters are
the best possible for series parallel graphs: if $c_p=O(1)$ then $c_s$ must be at least 2.

To generalize Lemma~\ref{lem:parallel} to all 2-connected graphs, we need to define the \emph{contractible} subgraph and the \emph{contraction} of a given graph as below.
\begin{definition}\label{def:contraction}
We say an induced subgraph $G_1$ of a 2-connected graph $G$  is \emph{contractible}, if there is a separating pair $\{u,v\}\subset V$ such that $G_1=(V_1,E_1)$ is a connected component of $G[V\backslash\{u,v\}]$. Moreover, if we replace $G_1$ by a weighted edge $e'$ with weight $w(e')=|V_1|$ between the nodes $u$ and $v$ in $G$ to obtain a smaller graph $G'$, we say $G$ is \emph{contracted} to $G'$.
\end{definition}
\begin{remark}
Notice that every contractible subgraph of a 2-connected graph $G$ can also be represented by a pseudo-path between its associated separating pair. We use this property in the proof of Theorem~\ref{th:2-connected}.
\end{remark}
\begin{figure}[t]
\centering
\begin{subfigure}[b]{0.6\textwidth}
\includegraphics[scale=0.5]{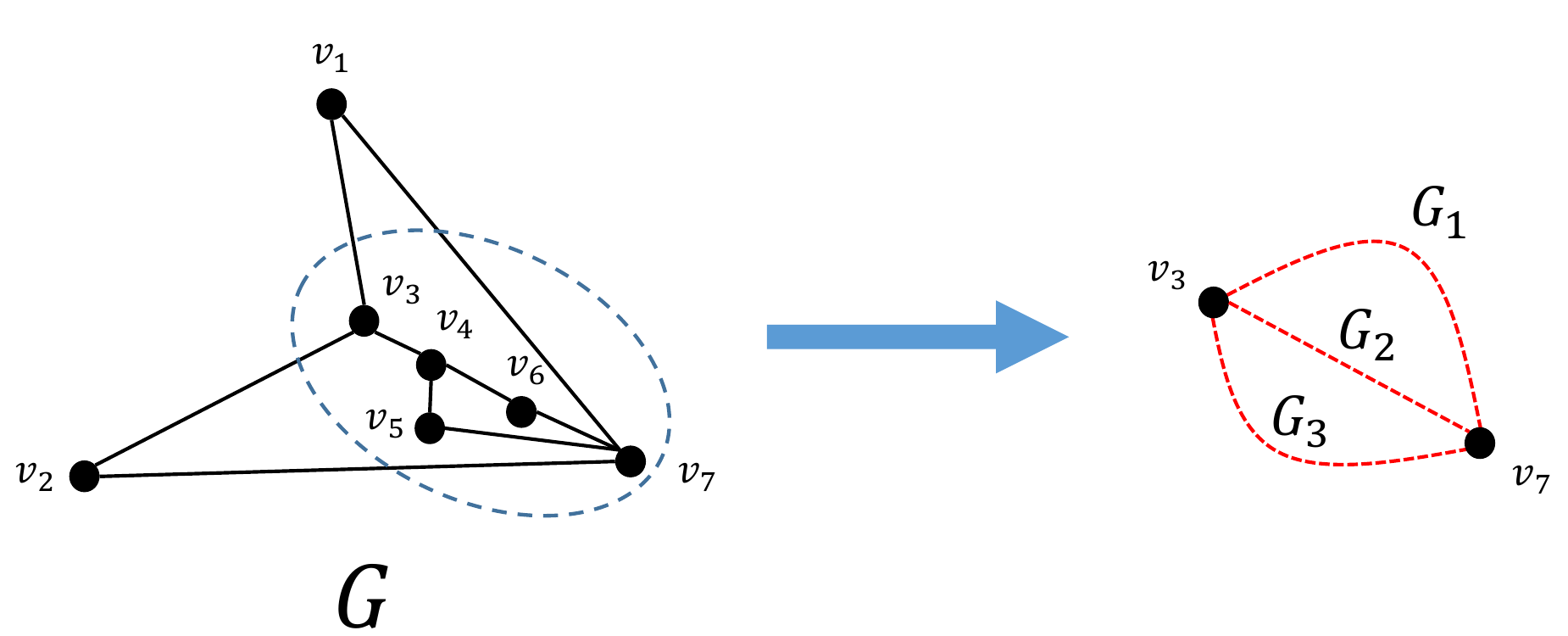}
\caption{}
\end{subfigure}
\begin{subfigure}[b]{\textwidth}
\includegraphics[scale=0.5]{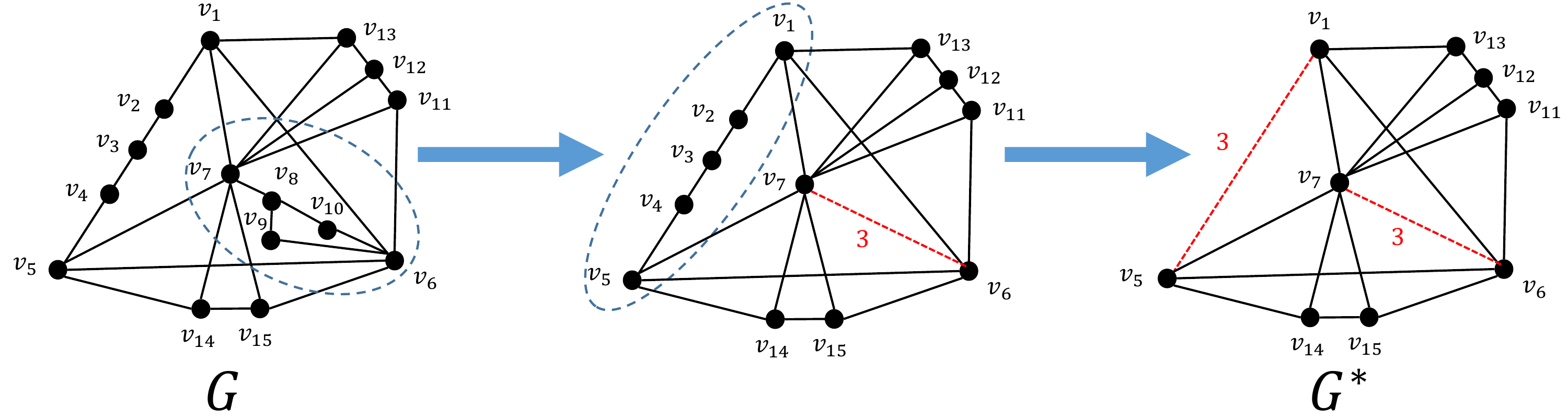}
\caption{}
\end{subfigure}
\caption{Lemma~\ref{lem:2-connected-pair}.}
\label{fig:lem_2-connected-pair}
\end{figure}
Using the notion of the graph contraction, in the following lemma, we show that to partition a 2-connected graph, we can reduce it to one of two cases: either $G$ can be considered as a graph with a set of short pseudo-paths between two nodes, or it can be contracted into a 3-connected graph as illustrated in Fig.~\ref{fig:lem_2-connected-pair}.
\begin{lemma}\label{lem:2-connected-pair}
In every 2-connected graph $G=(V,E)$, given an integer $q\geq 3$, one of the following cases holds, and we can determine which in $O(|E|)$ time:
\begin{enumerate}
\item There is a separation pair $\{u,v\}\subset V$ such that if $G_1,\dots,G_k$ are the connected components of $G[V\backslash \{u,v\}]$, for all $i$, $|V_i|< (q-1)|V|/q$.
\item After a set of contractions, $G$ can be transformed into a 3-connected graph $G^*=(V^*,E^*)$ with weighted edges representing contracted subgraphs such that for every $e^*\in E^*$, $w(e^*)<|V|/q$.
\end{enumerate}
\end{lemma}
\begin{proof}
If there is no separation pairs in $G$, then $G$ is 3-connected and there is nothing left to prove. So assume $\{u,v\}\subset V$ is a separation pair and  $G_1,\dots,G_k$ are the connected components of $G[V\backslash \{u,v\}]$. If $\forall i, |V_i|< (q-1)|V|/q$, we are done. So let's assume there is a connected component $G_j$ such that $|V_j|\geq (q-1)|V|/q$. Then for every $i\neq j$,  $G_i$ can be contracted and represented by an edge of weight less than $|V|/q$ between $u$ and $v$. Now, we repeat the process by considering the weight of the edges in the size of each connected component. An example for each case is shown in Fig.~\ref{fig:lem_2-connected-pair} for $q=3$. We can find either a suitable separation pair as in case 1 or a
suitable contracted graph $G^*$ as in case 2 in linear time using the Hopcroft-Tarjan
algorithm for finding the triconnected components \cite{hopcroft1973dividing}.
\end{proof}
 Using Lemma~\ref{lem:2-connected-pair} for $q=4$, then Lemma~\ref{lem:parallel}, and the idea of the proof for  Theorem~\ref{th:3-connected}, we can prove that when $G$ is 2-connected and all $p(i)=\pm 1$, the DBCP problem has a solution for $c_p=1$ and $c_s=3$. There are some subtleties in adapting Lemma~\ref{lem:cut} for this case to account for the fact that the edges of $G^*$ are now weighted, and the partition $(V_1', V_2')$ has to take into account the edge weights.
We find a suitable convex embedding of the 3-connected graph $G^*$
and then embed the nodes of the contracted pseudopaths appropriately
along the segments corresponding to the weighted edges.  Some care is needed
to carry out the argument of the 3-connected case, since as the line
tangent to the circle rotates, the order of the projections of many nodes may change at once, namely the nodes on an edge perpendicular to the rotating line.
The details of the proof are given in Appendix~\ref{sec:proof2}.
We have:
\begin{theorem}\label{th:2-connected}
If $G$ is 2-connected,  $\forall i, p(i)=\pm1$, then the DBCP problem has a solution for $c_p=1$ and $c_s=3$. Moreover, this solution can be found in polynomial time.
\end{theorem}

 Similar to  Corollary~\ref{cor:3-general}, the approach used in the proof of Theorem~\ref{th:2-connected}, can also be used for the case when the weights are arbitrary (not necessarily $\pm 1$) and  $p(V)\neq 0$. It is easy to verify that in this case, if $G$ is 2-connected, the DBCP problem has a
connected partition $(V_1,V_2)$ such that
$|p(V_1)-p(V)/2|, |p(V_2)-p(V)/2| \leq \max_{j\in V} |p(j)|$ and
$|V_1|, |V_2| \geq |V|/4$.
\begin{corollary}\label{cor:2-general}
If $G$ is 2-connected, then the DBCP problem (with general $p$ and not necessarily satisfying $p(V)= 0$) has a solution for $c_p=\max_{j\in V} |p(j)|$ and $c_s=3$. Moreover, this solution can be found in polynomial time.
\end{corollary}
\section{Graphs with Two Types of Nodes}\label{sec:blue_red}
Assume $G$ is a connected graph with nodes colored either red ($R\subseteq V$) or blue ($B\subseteq V$). Let $|V|=n$, $|R|=n_r$, and $|B|=n_b$.
If $G$ is 3-connected, set $p(i)=1$ if $i \in R$ and $p(i) =-1$ if $i \in B$.
Corollary~\ref{cor:3-general} implies then that there is always a connected partition $(V_1,V_2)$ of $V$  that splits both the blue and the red nodes evenly (assuming $n_r$ and $n_b$ are both even), i.e.,  such that $|V_1|=|V_2|$, $|R\cap V_1|=|R\cap V_2|$, and $|B\cap V_1|=|B\cap V_2|$. (If $n_r$ and/or $n_b$ are not even, then one side will
contain one more red or blue node.)
\begin{corollary}\label{cor:blue_red3}
Given a 3-connected graph $G$ with nodes colored either red ($R\subseteq V$) or blue ($B\subseteq V$). There is always a partition $(V_1,V_2)$ of $V$ such that $G[V_1]$ and $G[V_2]$ are connected, $|V_1|=|V_2|$, $|R\cap V_1|=|R\cap V_2|$, and $|B\cap V_1|=|B\cap V_2|$ (assuming $|R|$ and $|B|$ are both even). Such a partition can be
computed in polynomial time.
\end{corollary}

If $G$ is only 2-connected, we may not always get a perfect partition.
Assume wlog that $n_r \leq n_b$. If for every $v\in R$ and $u\in B$, we set $p(v)=1$ and $p(u)=-n_r/n_b$, Corollary~\ref{cor:2-general} implies that there is always a connected partition $(V_1,V_2)$ of $V$ such that  both $|(|R\cap V_1|-n_r/n_b |B\cap V_1|)|\leq 1$ and  $|(|R\cap V_2|-n_r/n_b |B\cap V_2|)|\leq 1$, and also $\max\{\frac{|V_1|}{|V_2|},\frac{|V_2|}{|V_1|}\}\leq 3$.
Thus, the ratio of red to blue nodes in each side $V_i$ differs from
the ratio $n_r/n_b$ in the whole graph by $O(1/n)$. Hence if the
numbers of red and blue nodes are $\omega(1)$, then the two types are
presented in both sides of the partition in approximately the same
proportion as in the whole graph.

\begin{corollary}
Given a 2-connected graph $G$ with nodes colored either red ($R\subseteq V$) or blue ($B\subseteq V$), and assume wlog $|R| \leq |B|$. We can always find
in polynomial time a partition $(V_1,V_2)$ of $V$
such that $G[V_1]$ and $G[V_2]$ are connected, $|V_1|,|V_2| \geq |V|/4$,
and the ratio of red to blue nodes in each side $V_i$ differs from the ratio $|R|/|B|$
in the whole graph by $O(1/n)$.
\end{corollary}
\section{Conclusion}\label{sec:conclusion}
In this paper, we introduced and studied the problem of partitioning a graph
into two connected subgraphs that satisfy simultaneouly two objectives:
(1) they balance the supply and demand within each side of the partition
(or more generally,  for the case of $p(V) \neq 0$, they split approximately equally
the excess supply/demand between the two sides), and (2) the two sides are large and
have roughly comparable size (they are both $\Omega(|V|)$).
We showed that for 2-connected graphs it is always possible to achieve both objectives at the same time, and for 3-connected graphs there is a partition
that is essentially perfectly balanced in both objectives. Furthermore,
these partitions can be computed in polynomial time.
This is a paradigmatic bi-objective balancing problem.
We observed how it can be easily used to find a connected partition of a graph with two types of nodes that is balanced with respect to the sizes of both types.
Overall, we believe that the novel techniques used in this paper can be applied to partitioning  heterogeneous networks in various contexts.

There are several interesting further directions that suggest themselves.
First, extend the theory and algorithms to find doubly balanced connected partitions to more than two parts.
Second, even considering only the supply/demand objective, does the analogue of the results
of Lov\'{a}z and Gyori~\cite{lovasz1977homology,gyori1976division} for the connected $k$-way partitioning
of $k$-connected graphs with respect to size (which corresponds to $p(i)=1$) extend to the
supply/demand case ($p(i)=\pm1$) for $k>3$? And is there a polynomial algorithm
that constructs such a partition?
Finally, extend the results of Section \ref{sec:blue_red} to graphs with more than two types of nodes,
that is, can we partition (under suitable conditions) a graph with several types of nodes to two (or more)
large connected subgraphs that preserve approximately the diversity (the proportions of the types) of the whole
population?

\newpage

\appendix
\section{Definitions and Background}\label{sec:pre}
In this section, we provide a short overview of the definitions and tools used in our work. Most of the graph theoretical terms used in this paper are relatively standard and borrowed from~\cite{bondy2008graph} and \cite{west2001introduction}. All the graphs in this paper are loopless.
\subsection{Terms from Graph Theory}
\noindent\textbf{Cutpoints and Subgraphs:} A \emph{cutpoint} of a connected graph $G$ is a node whose deletion results in a disconnected graph. 
Let $X$ and $Y$ be subsets of the nodes of a graph $G$. $G[X]$ denotes the subgraph of $G$ \emph{induced} by $X$. We denote by $E[X, Y]$ the set of edges of $G$ with one end in $X$ and the other end in $Y$. The neighborhood of a node $v$ is denoted $N(v)$.

\noindent\textbf{Connectivity:} The connectivity of a graph $G=(V,E)$ is the minimum size of a set $S\subset V$ such that $G\backslash S$ is not connected. A graph is $k$-connected if its connectivity is at least $k$.

\subsection{$st$-numbering of a Graph}\label{subsec:st-numbering}
Given any edge $\{s,t\}$ in a 2-connected graph $G$, an \emph{$st$-numbering} for $G$ is a numbering for the nodes in $G$ defined as follows~\cite{lempel1967algorithm}: the nodes of $G$ are numbered from 1 to $n$ so that $s$ receives number 1, node $t$ receives number $n$, and every node except $s$ and $t$ is adjacent both to a lower-numbered and to a higher-numbered node. It is shown in~\cite{even1976computing} that such a numbering can be found in $O(|V|+|E|)$.

\subsection{Series-Parallel Graphs}
A Graph $G$ is \emph{series-parallel}, with terminals $s$ and $t$, if it can be produced by a sequence of the following operations:
\begin{enumerate}
\item Create a new graph, consisting of a single edge between $s$ and $t$.
\item Given two series parallel graphs, $X$ and $Y$ with terminals $s_X,t_X$ and  $s_Y,t_Y$ respectively, form a new graph $G=P(X,Y)$ by identifying $s=s_X=s_Y$ and $t=t_X=t_Y$. This is known as the \emph{parallel composition} of $X$ and $Y$.
\item Given two series parallel graphs $X$ and $Y$, with terminals $s_X,t_X$ and $s_Y,t_Y$ respectively, form a new graph $G=S(X,Y)$ by identifying $s=s_X, t_X=s_Y$ and $t=t_Y$. This is known as the \emph{series composition} of $X$ and $Y$.
\end{enumerate}
It is easy to see that a series-parallel graph is 2-connected if, and only if, the last operation is a parallel composition.
\subsection{Nonseparating Induced Cycles and Ear Decomposition}\label{subsec:ear-decomposition}

Let $H$ be a subgraph of a graph $G$. An \emph{ear} of $H$ in $G$ is a nontrivial path in $G$ whose ends lie in $H$ but whose internal vertices do not. An ear decomposition of $G$ is a decomposition $G=P_0\cup\dots\cup P_k$ of the edges of $G$ such that $P_0$ is a cycle and $P_i$ for $i\geq 1$ is an ear of $P_0\cup P_1\cup\dots\cup P_{i-1}$ in $G$. It is known that every 2-connected graph has an ear decomposition (and vice-versa), and such a decomposition can be found in linear time.

A cycle $C$ is a \emph{nonseparating induced cycle} of $G$ if $G\backslash C$ is connected and $C$ has no chords. We say a cycle $C$ avoids a node $u$, if $u\notin C$.

\begin{theorem}[Tutte~\cite{tutte1963draw}] \label{th:tutte-3-connected}
Given a 3-connected graph $G(V,E)$ let $\{t,r\}$ be any edge of $G$ and let $u$ be any node of $G$, $r\neq u\neq t$. Then there is a nonseparating induced cycle of $G$ through $\{t,r\}$ and avoiding $u$.
\end{theorem}
Notice that since $G$ is 3-connected in the previous theorem, every node in $C$ has a neighbor in $G\backslash C$. Cheriyan and Maheshwari showed that the cycle in Theorem~\ref{th:tutte-3-connected} can be found in $O(E)$~\cite{cheriyan1988finding}. Moreover, they showed that any 3-connected graph $G$ has a nonseparating ear decomposition $G=P_0\cup\dots\cup P_k$ defined as follows:  Let $V_i=V(P_0)\cup V(P_1)\dots \cup V(P_i)$, let $G_i=G[V_i]$ and $\overline{G}_i=G[V\backslash V_i]$. We say that $G=P_0\cup P_1\cup\dots\cup P_k$ is an ear decomposition through edge $\{t,r\}$ and avoiding vertex $u$ if the cycle $P_0$ contains edge $\{t,r\}$, and the last ear of length greater than one, say $P_m$, has $u$ as its only internal vertex. An ear decomposition $P_0\cup P_1\dots\cup P_k$ of graph $G$ through edge $\{t,r\}$ and avoiding vertex $u$ is a \emph{nonseparating ear decomposition} if for all $i$, $0\leq i<m$, graph $\overline{G}_i$ is connected and each internal vertex of ear $P_i$ has a neighbor in $\overline{G}_i$. 
\begin{theorem}[Cheriyan and Maheshwari~\cite{cheriyan1988finding}]
Given an edge $\{t,r\}$ and a vertex $u$ of a 3-connected graph $G$, a nonseparating induced cycle of $G$ through $\{t,r\}$ and avoiding $u$, and
a nonseparating ear decomposition can be found in time $O(|V|+|E|)$.
\end{theorem}
\subsection{Partitioning of Graphs to Connected Subgraphs}
The following theorem is the main existing result in partitioning of graphs into connected subgraphs and is proved independently by Lov\'{a}z and Gyori~\cite{lovasz1977homology,gyori1976division} by different methods.
\begin{theorem}[Lov\'{a}z and Gyori~\cite{lovasz1977homology,gyori1976division}]
Let $G=(V,E)$ be a $k$-connected graph. Let $n=|V|,~ v_1,v_2,\dots,v_k\in V$ and let $n_1,n_2,\dots,n_k$ be positive integers satisfying $n_1+n_2+\dots+n_k=n$. Then, there exists a partition of $V$ into $(V_1,V_2\dots,V_k)$ satisfying $v_i\in V_i, |V_i|=n_i$, and $G[V_i]$ is connected for $i=1,2,\dots,k$.
\end{theorem}
Although the existence of such a partition has long been proved, there is no polynomial-time algorithm to find such a partition for $k>3$. For $k=2$, it is easy to find such partition using $st$-numbering. For $k=3$, Wada and Kawaguchi~\cite{wada1994efficient} provided an $O(n^2)$ algorithm using the nonseparating ear decomposition of $3$-connected graph.
\subsection{Convex Embedding of Graphs}\label{subsec:embedding}
In this subsection, we provide a short overview of the beautiful work by Linial, Lov\'{a}z, and Wigderson~\cite{linial1988rubber} on convex embedding of the $k$-connected graphs.
Let $Q=\{q_1,q_2,\dots,q_m\}$ be a finite set of points in $\mathbb{R}^d$. The convex hull $\text{conv}(Q)$ of $Q$ is the set of all points $\sum_{i=1}^m \lambda_i q_i$ with $\sum_{i=1} \lambda_i=1$. The rank of $Q$ is defined by $\text{rank}(Q)=1+\text{dim}(\text{conv}(Q))$. $Q$ is in general position if $\text{rank}(S)=d+1$ for every $(d+1)$-subset $S\subseteq Q$.
Let $G$ be a graph and $X\subset V$. A convex $X$-embedding of $G$ is any mapping $f:V\rightarrow \mathbb{R}^{|X|-1}$ such that for each $v\in V\backslash X$, $f(v)\in \mathrm{conv}(f(N(v)))$. We say that the convex embedding is in general position if the set $f(V)$ of the points is in general position.

\begin{theorem}[Linial, Lov\'{a}z, and Wigderson~\cite{linial1988rubber}]
Let $G$ be a graph on $n$ vertices and $1<k<n$. Then the following two conditions are equivalent:
\begin{enumerate}
\item $G$ is $k$-connected
\item For every $X\subset V$ with $|X|=k$, $G$ has a convex $X$-embedding in general position.
\end{enumerate}
\end{theorem}
Notice that the special case of the Theorem for $k=2$ asserts the existence of an $st$-numbering of a 2-connected graph. The proof of this theorem is inspired by physics. The embedding is found by letting the edges of the graph behave like ideal springs and letting its vertices settle. A formal summary of the proof ($1\rightarrow 2$) is as follows (for more details see~\cite{linial1988rubber}). For each $v_i\in X$, define $f(v_i)$ arbitrary in $\mathbb{R}^{k-1}$ such that $f(X)$ is in general position. Assign to every edge $(u,v)\in E$ a positive elasticity coefficient $c_{uv}$ and let $c\in \mathbb{R}^{|E|}$ be the vector of coefficients. It is proved in~\cite{linial1988rubber} that for almost any coefficient vector $c$, an embedding $f$ that minimizes  the potential function $P=\sum_{\{u,v\}\in E} c_{uv}\|f(u)-f(v)\|^2$ provides a convex $X$-embedding in general position ($\|.\|$ is the Euclidean norm). Moreover, the embedding that minimizes $P$ can be computed as follows,
\begin{equation*}
f(v)=\frac{1}{c_v}\sum_{u\in N(v)} c_{uv}f(u)~\text{for all}~v\in V\backslash X,
\end{equation*}
in which $c_v=\sum_{u\in N(v)} c_{uv}$. Hence, the embedding can be found by solving a set of linear equations in at most $O(|V|^3)$ time (or matrix multiplication time).
\section{Missing Proofs from Section~\ref{sec:BCPI}}\label{sec:proof1}
\begin{proof}[Proof of Proposition~\ref{lem:BCPI_hard}]
We use the proof of \cite[Theorem 2]{chataigner2007approximation} with a modest change. The reduction is from the X3C problem~\cite{papadimitriou1982complexity}, which is a variant of the \emph{Exact Cover by 3-sets} and defined as follows: Given a set $X$ with $|X|=3q$ and a family $C$ of 3-element subsets of $X$ such that $|C|=3q$ and each element of $X$ appears in exactly 3 sets of $C$, decide whether $C$ contains an exact cover for $X$.
Given an instance $(X,C)$ of $X3C$, let $G=(V,E)$ be the graph with the vertex set $V=X\cup C\cup\{a,b\}$ and edge set $E=\bigcup_{j=1}^{3q}[\{C_jx_i|x_i\in C_j\}\cup\{C_ja\}\cup\{C_jb\}]$. Set $p(a)=2q$, $p(b)=9q^2+q$, $p(C_j)=-1$, and $p(x_i)=-3q$. It is easy to verify that $C$ contains an exact cover for $X$ if and only if the BCPI problem has a solution such that $p(V_1)=p(V_2)=0$. This shows the first claim.

For the second claim, attach to nodes $a$, $b$, and the $x_i$s, paths of length $2q$, $9q^2+q$, and $3q$, respectively, and set the supply/demand values of $a$, $b$,  the $x_i$'s and the new nodes equal to $+1$ (for the paths for $a$ and $b$) or $-1$ (for the $x_i$'s). 
\end{proof}

\medskip

\begin{proof}[Proof of Proposition~\ref{lem:2_poly}]
Clearly, part (2) follows immediately from part (1) because in this case, $p(V_1), p(V_2)$
are integer and $\max_{j\in V}|p(j)|/2 =1/2$.
To show part (1), pick two arbitrary nodes $u,v\in V$ with $p(u) p(v)>0$. Since we want to separate $u$ and $v$, we can assume wlog that initially $\{u,v\}\in G$. Since $G$ is 2-connected, an $st$-numbering between nodes $u$ and $v$ as $u=v_1,v_2,\dots,v_n=v$ can be found in $O(|V|+|E|)$~\cite{even1976computing}. Define $V_1^{(i)}:=\{v_1,v_2,\dots,v_i\}$. It is easy to see that $p(V_1^{(1)})=p(u)>0$ and $p(V_1^{(n-1)})=-p(v)<0$. Hence, there must exist an index $1\leq i^*<n$ such that $|p(V_1^{(i^*)})|>0$ and $|p(V_1^{(i^*+1)})|\leq 0$. Since $|p(V_1^{(i)})-p(V_1^{(i+1)})|= |p(i^*+1)|$, either $|p(V_1^{(i^*)})|\leq |p(i^*+1)|/2$ or $|p(V_1^{(i^*+1)})|\leq |p(i^*+1)|/2$; Accordingly set $V_1=V_1^{(i^*)}$ or $V_1=V_1^{(i^*+1)}$. Let $V_2=V\backslash V_1$.  Hence, $(V_1,V_2)$ is a solution with $|p(V_1)|=|p(V_2)|\leq |p(i^*+1)/2|\leq \max_{j\in V}|p(j)|/2$. It is easy to see that $i^*$ can be found in $O(|V|)$.
\end{proof}

\medskip

\begin{proof}[Proof of Proposition~\ref{lem:3-connected-zero-p}]
Consider the case of general function $p$, and let $p_{\max} = \max_{j \in V} |p(j)|$.
We will show that we can find a solution such that $u \in V_1, v \in V_2, w \in V_3$ with $|p(V_1)|, |p(V_2)| \leq p_{\max}/2$.
Since $|p(V_3)| = |p(V_1) + p(V_2)|$  (recall $p(V)=0$),
this implies that $|p(V_3)| \leq p_{\max}$, and hence
$|p(V_1)|+ |p(V_2)| +|p(V_3)| \leq 2p_{\max}$.
Furthermore, if $ p(i) = \pm 1$ for all $i \in V$, hence $p_{\max}=1$,
then $|p(V_1)|, |p(V_2)| \leq p_{\max}/2$ implies that 
$p(V_1) =p(V_2)=0$, and therefore also $p(V_3) = 0$.
Thus, both claims will follow.

Assume $u,v,w\in V$ and $p(u), p(v),p(w)>0$ (the case of negative $p(u), p(v),p(w)$
is symmetric). 
Since we want to separate $u$ from $v$, we can assume without loss of generality that $\{u,v\}\in E$. Using~\cite{cheriyan1988finding}, there is a non-separating ear decomposition through the edge $\{u, v\}$ and avoiding node $w$.  
Ignore the ears that do not contain any internal nodes, and let $Q_0\cup Q_1\cup\dots\cup Q_r$ be the decomposition consisting of the ears with nodes; we have $w\in Q_r$. Let $V_i=V(Q_0)\cup V(Q_1)\dots \cup V(Q_i)$, let $G_i=G[V_i]$ and $\overline{G}_i=G[V\backslash V_i]$. We distinguish two cases, depending on whether $p(V_0)\leq0$
or $p(V_0)>0$.
\begin{itemize}
\item[(i)] If $p(V_0)\leq0$, then consider an $st$-numbering between $u$ and $v$ in $V_0$, say $u=v_1,v_2,\dots,v_s=v$. Define $V_0^{(i)}=\{v_1,v_2,\dots,v_i\}$. Since $p(u),p(v)>0$ and $p(V_0)\leq0$, there must exist indices $1\leq i^*\leq j^*<s$ such that $p(V_0^{(i^*)})>0,~p(V_0^{(i^*+1)})\leq0$ and $p(V_0\backslash V_0^{(j^*+1)})>0,~p(V_0\backslash V_0^{(j^*)})\leq0$.
    \begin{itemize}
    \item[(a)] If $i^*=j^*$, since $p(V_0^{(i^*)})+ p(v_{i^*+1})+p(V_0\backslash V_0^{(i^*+1)})=p(V_0)<0$, we have $ p(V_0^{(i^*)})+p(V_0\backslash V_0^{(i^*+1)})\leq  |p(v_{i^*+1})|$.
    Now, one of the following three cases happens:
    \begin{itemize}
    \item[-] If $p(V_0^{(i^*)})\leq |p(v_{i^*+1})|/2$ and $p(V_0\backslash V_0^{(i^*+1)})\leq|p(v_{i^*+1})|/2$, then it is easy to see that $V_1=V_0^{(i^*)}$, $V_2=V_0\backslash V_0^{(i^*+1)}$, and $V_3=V\backslash (V_1\cup V_2)$ is a good partition.
    \item[-] If $p(V_0^{(i^*)})> |p(v_{i^*+1})|/2$ and $p(V_0\backslash V_0^{(i^*+1)})\leq|p(v_{i^*+1})|/2$, then $p(V_0^{(i^*)})+p(v_{i^*+1})=p(V_0^{(i^*+1)})\leq |p(v_{i^*+1})|/2$. Hence, $V_1=V_0^{(i^*+1)}$, $V_2=V_0\backslash V_0^{(i^*+1)}$, and $V_3=V\backslash V_0$ is a good partition.
    \item[-] If $p(V_0^{(i^*)})\leq |p(v_{i^*+1})|/2$ and $p(V_0\backslash V_0^{(i^*+1)})>|p(v_{i^*+1})|/2$, then $p(V_0\backslash V_0^{(i^*+1)})+p(v_{i^*+1})= p(V_0\backslash V_0^{(i^*)})\leq|p(v_{i^*+1})|/2$. Hence, $V_1=V_0^{(i^*)}$, $V_2=V_0\backslash V_0^{(i^*)}$, and $V_3=V\backslash V_0$ is a good partition.
    \end{itemize}
    \item[(b)] If $i^*<j^*$, then either $p(V_0^{(i^*)})\leq|p(v_{i^*+1})|/2$ or $|p(V_0^{(i^*+1)})|\leq|p(v_{i^*+1})|/2$, accordingly set $V_1=V_0^{(i^*)}$ or $V_1=V_0^{(i^*+1)}$. Similarly, either $p(V_0\backslash V_0^{(j^*+1)})\leq |p(v_{j^*+1})|/2$ or $|p(V_0\backslash V_0^{(j^*)})|\leq |p(v_{j^*+1})|/2$, so accordingly set $V_2=V_0\backslash V_0^{(j^*+1)}$ or $V_2=V_0\backslash V_0^{(j^*)}$. Set $V_3=V\backslash (V_1\cup V_2)$. It is easy to check that $(V_1,V_2,V_3)$ is a good partition.
    \end{itemize}
\item[(ii)] If $p(V_0)>0$, then since $p(w)>0$ and therefore $p(V_{r-1})<0$, there must exist an index $0\leq j<r-1$ such that $p(V_j)>0$ and $p(V_{j+1})\leq0$. Consider an $st$-numbering between $u$ and $v$ in $G[V_j]$ as $u=v_1,v_2,\dots,v_s=v$ and define $V_j^{(i)}=\{v_1,v_2,\dots,v_i\}$. The 
ear $Q_{j+1}$ is a path of new nodes $q_1,q_2,\dots, q_t$
attached to two (distinct) nodes $v_x, v_y$ of $G[V_j]$ through
edges $\{v_x,q_1\},\{q_t,v_y\}\in E$; assume wlog that
$1\leq x<y\leq s$.
For simplicity, we will use below $Q_{j+1}$ to denote also the
set $\{q_1,q_2,\dots, q_t\}$ of internal (new) nodes of the ear.
Also define $Q_{j+1}^{(i)}=\{q_1,q_2,\dots,q_i\}$ and $Q_{j+1}^{(0)}=\emptyset$. 
One of the following cases must happen:
    \begin{itemize}
    \item[(a)] Suppose there is an index $1\leq i^*<(y-1)$ such that $p(V_j^{(i^*)})>0$ and $p(V_j^{(i^*+1)})\leq0$ or there is an index $x+1<i^*< s$ such that $p(V_j\backslash V_j^{(i^*-1)})>0$ and $p(V_j\backslash V_j^{(i^*)})\leq0$. 
         Let's assume there is an index $1\leq i^*< (y-1)$, such that $p(V_j^{(i^*)})>0$ and $p(V_j^{(i^*+1)})\leq0$ (the other case is exactly similar). Then either $p(V_j^{(i^*)})\leq|p(v_{i^*+1})|/2$ or $|p(V_j^{(i^*+1)})|\leq|p(v_{i^*+1})|/2$, accordingly set either $V_1=V_j^{(i^*)}$ or $V_1=V_j^{(i^*+1)}$. Set $V_2'=V_j\backslash V_1$. One of the following cases happens:
        \begin{itemize}
        \item[-] If $V_1=V_j^{(i^*)}$ and $p(V_2')\leq 0$, then since $p(V_j^{(i^*+1)})\leq0$, we have $p(V_j\backslash V_j^{(i^*+1)})>0$. Hence, $p(V_2'\backslash \{v_{i^*+1}\})>0$. So, it is either $|p(V_2')|\leq |p(v_{i^*+1})|/2$ or $p(V_2'\backslash \{v_{i^*+1}\})\leq |p(v_{i^*+1})|/2$. Now if $p(V_2'\backslash \{v_{i^*+1}\})\leq |p(v_{i^*+1})|/2$, since also $p(V_1)\leq |p(v_{i^*+1})|/2$, $p(V_j)\leq 0$ which contradicts with the assumption. Therefore, $|p(V_2')|\leq |p(v_{i^*+1})|/2$. Set $V_2=V_2'$ and $V_3=V\backslash (V_1\cup V_2)$. It is easy to check that $(V_1,V_2,V_3)$ is a good partition.
        \item[-] If $V_1=V_j^{(i^*)}$ and $p(V_2')> 0$, then since $p(V_j\cup Q_{j+1})<0$, there is an index $0< t^*\leq t$, such that $p(V_2'\cup (Q_{j+1}\backslash Q_{j+1}^{(t^*)}))>0$ and $p(V_2'\cup (Q_{j+1}\backslash Q_{j+1}^{(t^*-1)}))\leq 0$. Hence, either $p(V_2'\cup (Q_{j+1}\backslash Q_{j+1}^{(t^*)}))\leq |p(q_{t^*})|/2$ or $|p(V_2'\cup (Q_{j+1}\backslash Q_{j+1}^{(t^*-1)}))|\leq |p(q_{t^*})|/2$, accordingly set $V_2=V_2'\cup (Q_{j+1}\backslash Q_{j+1}^{(t^*)})$ or $V_2=V_2'\cup (Q_{j+1}\backslash Q_{j+1}^{(t^*-1)})$. Set $V_3=V\backslash (V_1\cup V_2)$. It is easy to see that $(V_1,V_2,V_3)$ is a good partition.
        \item[-] If $V_1=V_j^{(i^*+1)}$, then since $p(V_1)\leq 0$, we have $p(V_2')>0$. The rest is exactly like the previous case when $V_1=V_j^{(i^*)}$ and $p(V_2')> 0$.

        \end{itemize}
    \item[(b)] Suppose that for every $1\leq i<y$, $p(V_j^{(i)})>0$ and for every $x<i< s$, $p(V_j\backslash V_j^{(i)})>0$. Set $V_1'=V_j^{(y-1)}$ and $V_2'=V_j\backslash V_1'$. Based on the assumption $p(V_1'),p(V_2')>0$. Since $p(V_{j+1})\leq 0$, there are indices $0\leq i^*\leq j^*<t$ such that $p(V_1'\cup Q_{j+1}^{(i^*)})>0$, $p(V_1'\cup Q_{j+1}^{(i^*+1)})\leq0$ and $p(V_2'\cup (Q_{j+1}\backslash Q_{j+1}^{(j^*+1)}))>0$, $p(V_2'\cup (Q_{j+1}\backslash Q_{j+1}^{(j^*)}))\leq0$. The rest of the proof is similar to case (i) when $p(V_0)\leq 0$.
    \end{itemize}

\end{itemize}
\end{proof} 
\section{Missing Proofs from Section~\ref{sec:BPGI}}\label{sec:proof2}
\begin{proof}[Proof of Lemma~\ref{lem:embedding}]
Set $X=\{v,u,w\}$. Using~\cite{linial1988rubber}, $G$ has a convex $X$-embedding in $\mathbb{R}^2$ in general position with mapping $f:V\rightarrow \mathbb{R}^2$ such that $f(u)=(0,0)$, $f(v)=(1,0)$, and $f(w)=(0,1)$.  In the $X$-embedding of the nodes, we have a freedom to set the elasticity coefficient vector $\vec{c}$ to anything that we want (except a measure zero set of vectors). So for any edge $\{i,j\}\in G[V_1']\cup G[V_2']$, set $c_{ij}=g$; and for any $\{i,j\} \in E[V_1',V_2']$, set $c_{ij}=1$.
Assume $\mathcal{L}_1$ is the line $y=0.5$, $\mathcal{L}_2$ is the line $x+2y=1$, and $\mathcal{L}_3$ is the line $x=y$.

First, we show that there exist a $g$ for which all the nodes in $V_1'$ will be embedded above the line $\mathcal{L}_1$. To show this, from~\cite{linial1988rubber}, we know the embedding is such that it minimizes the total potential $P(f,\vec{c})= \sum_{\{i,j\}\in E} c_{ij}\|f(i)-f(j)\|^2$. Notice that we can independently minimize $P$ on $x$-axis values and $y$-axis values as below:
\begin{eqnarray*}
\min_{f} P &=& \min_{f_1} P_x +\min_{f_2} P_y\\
&=&\min_{f_1} \sum_{\{i,j\}\in E} c_{ij} (f_1(i)-f_1(j))^2+ \min_{f_2} \sum_{\{i,j\}\in E} c_{ij} (f_2(i)-f_2(j))^2
\end{eqnarray*}

Now, notice that if we place all the nodes in $V_1'$ at point (0,1) and all the nodes in $V_2'$ on the line $uv$, then $P_y\leq |E|$. Hence, if $f_2$ minimizes $P_y$, then $P_y(f_2,c)\leq |E|$. Set $g \geq 4|V|^2 |E|$. We show that if $f_2$ minimizes $P_y$, then for all edges $\{i,j\}\in G[V_1']\cup G[V_2']$, $(f_2(i)-f_2(j))^2\leq1/(4|V|^2)$. By contradiction, assume there is an edge $\{i,j\}\in G[V_1']\cup G[V_2']$ such that $(f_2(i)-f_2(j))^2>1/(4|V|^2)$. Then, $c_{ij}(f_2(i)-f_2(j))^2=g (f_2(i)-f_2(j))^2 >|E|$. Hence, $P_y(f_2,c)> |E|$ which contradicts with the fact the $f_2$ minimizes $P_y$. Therefore, if $g\geq4|V|^2 |E|$, then for all $\{i,j\}\in G[V_1']\cup G[V_2']$, $|f_2(i)-f_2(j)|\leq1/(2|V|)$. Now, since $G[V_1']$ is connected, all the nodes in $V_1'$ are connected to $w$ with a path of length (in number of hops) less than $|V|-1$. Hence, using the triangle inequality, for all $i\in V_1'$:
\begin{equation*}
|f_2(w)-f_2(i)|\leq (|V|-1)/(2|V|)<1/2\Rightarrow |1-f_2(i)|< 1/2,
\end{equation*}
which means that all the nodes in $V_1'$ are above $\mathcal{L}_1$.

With the very same argument, if $g\geq t^2 |V|^2 |E|$, then for all $i\in V_2'$, $f_2(i)<1/t$.

Now, we want to prove that there is a $g$ such that all the nodes in $V_2'$ will be embedded below the lines $\mathcal{L}_2$ and $\mathcal{L}_3$. Define $n_1(i):=|N(i)\cap V_1'|$ and $n_2(i):=|N(i)\cap V_2'|$. From~\cite{linial1988rubber}, we know the embedding is such that for all $i\in V\backslash\{u,v,w\}$, $f(i)=1/c_i\sum_{j\in N(i)}c_{ij}f(j)$, where $c_j=\sum_{j\in N(i)}c_{ij}f(j)$. Since $G[V'_2]$ is connected, for any $i\in V_2'$ there is a path $i=v_1,v_2,\dots,v_r=v$ in $V_2'$. Using this ordering:
\begin{eqnarray*}
&&\begin{cases}
f_1(v_j)\geq \frac{1}{n_2(v_j)g+n_1(v_j)}g f_1(v_{j+1})\geq (1/|V|)f_1(v_{j+1}),& \forall j\in \{1,\dots,r-1\}\\
f_1(v_r)=f_1(v)=1
\end{cases}
\\&&\Rightarrow \forall i\in V_2'\backslash\{u,v\}:~ f_1(i)\geq (1/|V|)^{r}\geq(1/|V|)^{|V|}.
\end{eqnarray*}
On the other hand, from the previous part, if we set $g\geq |V|^{2|V|+2}|E|$, then for all $i\in V_2'$, $f_2(i)\leq (1/|V|)^{|V|}$. Hence, for all $i\in V_2'$, $f_2(i)\leq f_1(i)$, which means that all the nodes in $V_2'$ will be placed below the line $\mathcal{L}_3$.

With the very same idea, we show that there exist a $g$ for which all the nodes in $V_2'$ will be placed below the line $\mathcal{L}_2$. Since $G[V'_2]$ is connected, for any $i\in V_2'$ there is a path $u=u_1,v_2,\dots,u_t=i$ in $V_2'$. Notice that for all $i\in V\backslash\{u,v,w\}$, $1-f_1(i)=1/c_i\sum_{j\in N(i)}c_{ij}(1-f_1(j))$. Hence, since $\forall j\in V: f_1(j)\leq 1$, we have,
\begin{eqnarray*}
&&\begin{cases}
1-f_1(u_j)\geq \frac{1}{n_2(u_j)g+n_1(u_j)}g (1-f_1(u_{i-1}))\geq (1/|V|)(1-f_1(u_{i-1})),& \forall j\in \{2,\dots,t\}\\
1-f_1(u)=1-f_1(u_1)=1
\end{cases}
\\&& \Rightarrow \forall i\in V_2'\backslash\{u,v\}:~1-f_1(i)\geq (1/|V|)^{t}\geq (1/|V|)^{|V|}.
\end{eqnarray*}
From the previous part, if we set $g\geq 4 |V|^{2|V|+2}|E|$, then for all $i\in V_2'$, $f_2(i)\leq 1/2(1/|V|)^{|V|}$. Hence, for $i\in V_2'$, $f_1(i)+2f_2(i)\leq 1$, which means that all the nodes in $V_2'$ will be placed below the line $\mathcal{L}_3$. Therefore, if we set $g\geq 4|V|^{2|V|+2}|E|$, then we will get an embedding as depicted in Fig.~\ref{fig:3-connected}.
Note that a polynomial number of bits suffices for $g$.

Notice that if $\vec{c}$ is a ``good" vector, then so is $\vec{c}+\vec{\epsilon}$ in which $\vec{\epsilon}$ is a vector with very small Euclidean norm. Hence, we can always find a ``good" vector $\vec{c}$ which result in a $X$-embedding in general position.
\end{proof}

\medskip

\begin{proof}[Proof of Theorem~\ref{th:3-connected}]
Assume that $|V|\equiv0(\mathrm{mod}~4)$; the proof for the case $|V|\equiv2(\mathrm{mod}~4)$ is similar. Using Lemma~\ref{lem:cut}, we can find $\{u,v,w\}\in V$ and a partition $(V_1',V_2')$ of $V$ with properties described in the Lemma. Set $X=\{u,v,w\}$. Using Lemma~\ref{lem:embedding},  we can find a convex $X$-embedding of $G$  in general position with properties described in the Lemma as depicted in Fig.~\ref{fig:3-connected}.
The rest of the proof is very similar to the proof of Lemma~\ref{lem:3-connected-tri}. We consider again a circle $\mathcal{C}$ around $f(u),f(v),f(w)$ in $\mathbb{R}^2$ as shown in Fig.~\ref{fig:3-connected}. Also consider a directed line $\mathcal{L}$ tangent to the circle $C$ at point $A$. If we project the nodes of $G$ onto the line $\mathcal{L}$, this time the order of the nodes projection gives an $st$-numbering between the first and the last node only if $u$ and $v$ are the first and last node. However, if we set $V_1$ to be the $|V|/2$ nodes whose projections come first and $V_2$ are the $|V|/2$ nodes whose projections come last, then $G[V_1]$ and $G[V_2]$ are both connected even when $u$ and $v$ are not the first and last nodes. The reason lies on the special embedding that we considered here. Assume for example $w$ and $v$ are the first and the last projected nodes, and $V_1$ and $V_2$ are set of the $|V|/2$ nodes which projections come first and last, respectively. Two cases might happen:
\begin{itemize}
\item[(i)] If $u,w\in V_1$ and $v\in V_2$, then since $\{u,w\}\in E$, both $G[V_1]$ and $G[V_2]$ are connected because of the properties of the embedding.
\item[(ii)] If $w\in V_1$ and $u,v\in V_2$, since $|V_2'|\leq |V|/2$ and $|V_2|=|V|/2$, then either $V_2=V_2'$ or $V_2 \cap  V_1' \neq \emptyset$.
If $V_2= V_2'$, and hence $V_1=V_1'$ then there is nothing to prove.
So assume there is a node $z\in V_2 \cap V_1'$.
From the properties of the embedding, the triangle $\{z,u,v\}$
contains all the nodes of $V_2'$.
Since $\{z,u,v\}\in V_2$, and $V_2$ contains all the nodes that are on a same side of a halfplane, we should also have $V_2'\subset V_2$. Now, from the properties of the embedding, it is easy to see that every node in $V_2$ has a path either to $u$ or $v$. Since $V_2'\subset V_2$, there is also a path between $u$ and $v$. Thus, $G[V_2]$ is connected. From the properties of the embedding, $G[V_1]$ is connected as before.
\end{itemize}

The rest of the proof is exactly the same as the proof of Lemma~\ref{lem:3-connected-tri}. We move $\mathcal{L}$ from being tangent at point $A$ to point $B$ ($AB$ is a diameter of the circle $\mathcal{C}$) and consider the resulting partition. Notice that if at point $A$, $p(V_1)>0$, then at point $B$ since $V_1$ and $V_2$ completely switch places compared to the partition at point $A$,  $p(V_1)<0$. Hence, as we move $\mathcal{L}$ from being tangent at point $A$ to point $B$ and keep it tangent to the circle, in the resulted partitions, $p(V_1)$ goes from some positive value to a negative value. Notice that the partition $(V_1,V_2)$ changes only if $\mathcal{L}$ passes a point $D$ on the circle such that at $D$, $\mathcal{L}$ is perpendicular to a line that connects $f(i)$ to $f(j)$ for $i,j\in V$.  Now, since the embedding is in general position, there are exactly two points on every line that connects two points $f(i)$ and $f(j)$, so $V_1$ changes at most by one node leaving $V_1$ and one node entering  $V_1$. Hence, $p(V_1)$ changes by either $\pm 2$ or $0$ value at each change. Now, since $|V|\equiv 0 (\mathrm{mod}~ 4)$, $p(V_1)$ has an even value in all the resulting partitions. Therefore, as we move $\mathcal{L}$ from being tangent at point $A$ to point $B$, there should be a point $D$ such that in the resulted partition $p(V_1)=p(V_2)=0$.

\end{proof}

\medskip

\begin{proof}[Proof of Corollary~\ref{cor:serpar}]
Every series-parallel graph $G$ has a separation pair $\{u,v\}$
such that every connected component of $G[V\backslash \{u,v\}]$ has less that $2|V|/3$ nodes,
and furthermore, such a separation  pair can be found in linear time.
To see this, consider the derivation tree $T$ of the construction of $G$.
The root of $T$ corresponds to $G$, the leaves correspond to the edges, and every internal node
$i$ corresponds to a subgraph $G_i=(V_i,E_i)$ that is the series or parallel composition of the
subgraphs corresponding to its children.
Starting at the root of $T$, walk down the tree following always the edge to the child corresponding to
a subgraph with the maximum number of nodes until the number of nodes becomes $\leq 2|V|/3$.
Thus, we arrive at a node $i$ of the tree such that $|V_i| > 2|V|/3$
and $|V_j| \leq 2|V|/3$ for all children $j$ of $i$.
Let $u_i, v_i$ be the terminals of $G_i$. Note that $u_i, v_i$ separate all
the nodes of $G_i$ from all the nodes that are not in $G_i$.
Since $|V_i| > 2|V|/3$, we have $|V \backslash V_i| < |V|/3$.
If $G_i$ is the parallel composition of the graphs corresponding to the children of $i$,
then the separation pair $\{u_i, v_i\}$ has the desired property,
i.e. all the components of $G[V\backslash \{u,v\}]$ have less than $2|V|/3$ nodes.

Suppose $G_i$ is the series composition of the graphs $G_j$, $G_k$
corresponding to the children $j,k$ of $i$, and let $w$ be the common
terminal of $G_j$, $G_k$; thus, $G_i$ has terminals $u_i, w$,
and $G_k$ has terminals $w, v_i$.
Assume wlog that $|V_j| \geq |V_k|$.
Then $|V|/3 < |V_j| \leq 2|V|/3$.
The pair $\{u_i,w\}$ of terminals of $G_j$ separates all the nodes
of $V_j \backslash \{u_i,w\}$ from all the nodes of $V \backslash V_j$,
and both these sets have less than $2|V|/3$ nodes.
Thus, $\{u_i,w\}$ has the required property.
\end{proof}

\medskip

\begin{proof}[Proof of Theorem~\ref{th:2-connected}]
Using Lemma~\ref{lem:2-connected-pair} for $q=4$, we consider two cases:
\begin{itemize}
\item[(i)] There is a separation pair $\{u,v\}\in V$ such that if $G_1,\dots,G_k$ are the connected components of $G\backslash \{u,v\}$, for any $i$, $|V_i|< 3|V|/4$. In this case Lemma~\ref{lem:parallel} for $q=4$ proves the theorem.

\item[(ii)] After a set of contractions, $G$ can be transformed into a 3-connected graph $G^*=(V^*,E^*)$ with weighted edges  such that for any edge $e^*\in E^*$, $w(e^*)<|V|/4$. In this case the proof is similar to the proof of Theorem~\ref{th:3-connected}. Notice that if $G^*$ contains a triangle then the proof is much simpler as in the proof of Lemma~\ref{lem:3-connected-tri} but here to avoid repetition, we use the approach in the proof of Theorem~\ref{th:3-connected} and prove the theorem once for all cases of $G^*$.

    Recall that for each edge in $G^*$, its weight represents the number of nodes in the contracted subgraph that it represents. So if $e^*\in E^*$ represents an induced subgraph of $G$ with $l$ nodes, then $w(e^*)=l$; and if $e^*\in E\cap E^*$ then $w(e^*)=0$. Using Lemma~\ref{lem:cut}, we can find $\{u,v,w\}\in V^*$ and a partition $(V_1^*,V_2^*)$ of $V^*$ with properties described in the Lemma. Since in $G^*$ edges have weights that actually represent nodes in $G$, we want a partition such that we also have $|V^*_2|+\sum_{e\in G[V_2^*]}w(e)\leq |V|/2$. Notice that we can modify the proof of Lemma~\ref{lem:cut} to take into account the weights for the edges and find a partition such that $|V^*_2|+\sum_{e\in G[V_2^*]}w(e)\leq |V|/2$ as follows. Again, using the algorithm presented in~\cite{cheriyan1988finding}, we can find a non-separating proper cycle $C_0$ in $G^*$ such that every node in $C_0$ has a neighbor in $G^*\backslash C_0$. We consider three cases:
    \begin{itemize}
    \item[(a)] If $|C_0|+\sum_{e\in G[C_0]}w(e)\leq |V|/2+1$, then the proof is as in the proof of Lemma~\ref{lem:cut}. Select any three consecutive nodes $(u,w,v)$ of $C_0$ and set $V_2^*=C_0\backslash \{w\}$ and $V_1^*=V^*\backslash V_2^*$.
    \item[(b)] If $|C_0|>|V^*|/2$ and $|C_0|+\sum_{e\in G[C_0]}w(e)> |V|/2+1$, then as in the proof of Lemma~\ref{lem:cut}, since $|C_0|>|V^*\backslash C_0|$ and every node is $C_0$ has a neighbor in $V^*\backslash C_0$, there exist a node $w\in V^*\backslash C_0$ such that $|N(w)\cap C_0|\geq 2$. Select two nodes $u,v\in N(w)\cap C_0$. There exists a path $P$ in $C_0$ between $u$ and $v$ such that $|P|+\sum_{e\in G[P]}w(e)<|V|/2-1$. Set $V_2^*=\{u,v\}\cup P$ and $V_1^*=V^*\backslash V_2^*$.
    \item[(c)] If $|C_0|\leq |V^*|/2$ and $|C_0|+\sum_{e\in G[C_0]}w(e)> |V|/2+1$, then if there exist a node $w\in V^*\backslash C_0$ such that $|N(w)\cap C_0|\geq 2$, the proof is as in the previous part. So assume for all $w\in V^*\backslash C_0$, $|N(w)\cap C_0|\leq 1$. Assume $w_1,w_2,\dots,w_t\in V^*\backslash C_0$ are all the nodes with $|N(w_i)\cap C_0|=1$. We show that there is a $1\leq i\leq t$, such that $w_i$ is not a cut-point of $G^*[V^*\backslash C_0]$. Let $T$ be a spanning tree of $G^*[V^*\backslash C_0]$. If there is a $1\leq i\leq t$, such that $w_i$ is a leaf of $T$, then $w_i$ is not a cut-point of $G^*[V^*\backslash C_0]$ and there is nothing left to prove. So assume none of $w_i$s is a leaf of $T$. Suppose $1\leq p,q\leq t$ are such that the path between $w_p$ and $w_q$ in $T$ is the longest between all pairs of $w_i$s. We show that $w_p$ and $w_q$ cannot be cut-points of $G^*[V^*\backslash C_0]$. Assume for example $w_p$ is a cut-point of $G^*[V^*\backslash C_0]$ and $G_{\bar{q}}^*$ is a connected component of $G^*[V^*\backslash C_0]\backslash\{w_p\}$ such that $w_q\notin G_{\bar{q}}^*$. Since the path between $w_p$ and $w_q$ in $T$ is the longest between all pairs of $w_i$s, $\forall 1\leq i\leq t: w_i\notin G_{\bar{q}}^*$, otherwise we can find a longer path from $w_q$ to some other $w_i$. On the other hand, if $\forall 1\leq i\leq t: w_i\notin G_{\bar{q}}$, in $G^*\backslash \{w_p\}$ the cycle $C_0$ is disconnected from $G_{\bar{q}}^*$ which contradicts with the 3-connectedness of $G^*$. Hence, there should be at least a noncut-point $w\in V^*\backslash C_0$ with $|N(w)\cap C_0|=1$. Since $G^*$ is 3-connected, each node has degree at least 3. Hence, there are two nodes $u,v\in N(w)\cap (V^*\backslash C_0)$. Set $V_1^*=C_0\cup\{w\}$ and $V_2^*=V^*\backslash V_1^*$. $(V_1^*,V_2^*)$ is a desirable partition for $V^*$.
    \end{itemize}
    Hence, we can find $\{u,v,w\}\in V^*$ and a partition $(V_1^*,V_2^*)$ of $V^*$ with properties described in the Lemma~\ref{lem:cut} as well as having $|V^*_2|+\sum_{e\in G[V_2^*]}w(e)\leq |V|/2$. Set $X=\{u,v,w\}$. Using Lemma~\ref{lem:embedding}, $G^*$ has a convex $X$-embedding in general position like $f^*:V^*\rightarrow \mathbb{R}^2$ as described in the lemma and depicted in Fig.~\ref{fig:3-connected}. Now, from this embedding, we get a convex $X$-embedding for $G$ like $f:V\rightarrow \mathbb{R}^2$ as follows. For any $i\in V\cap V^*$, $f(i)=f^*(i)$. For any edge $\{i,j\}\in E^*$ such that $\{i,j\}$ represents an induced subgraph of $G$, we represent it by a pseudo-path between $i$ and $j$ in $G$ like $P$ and place the nodes in $P$ in order on random places on the line segment that connects $f(i)$ to $f(j)$. If the edge $\{i,j\}\in E^*$ is between a node in $V_1^*$ and a node in $V_2^*$ and represents a pseudo-path $P$ in $G$ , we place the nodes in $P$ in order on random places on the segment that connects $f(i)$ to $f(j)$ but above the line $\mathcal{L}_1$. Hence, by this process, we get a convex $X$-embedding for $G$ which is in general position (almost surely) except for the nodes that are part of a pseudo-path (which we know have length less than $|V|/4$). Notice that if $V_2'\subset V$ contains all the nodes in $V_2^*$ and all the nodes that are part of a pseudo-path between the nodes of $V_2^*$ (represented by weighted edges in $G^*$) in $G$, and $V_1'=V\backslash V_2'$, then $(V_1',V_2')$ is partition of $G$ with all the properties of Lemma~\ref{lem:cut}. Moreover, the $X$-embedding $f$ of $G$ that we derived from $f^*$ for $G^*$, has all the properties of Lemma~\ref{lem:embedding} for the partition $(V_1',V_2')$.

    The rest of the proof is similar to the proof of Theorem~\ref{th:3-connected}. We consider again a circle $\mathcal{C}$ around $f(u),f(v),f(w)$ in $\mathbb{R}^2$ as shown in Fig.~\ref{fig:3-connected}. Also consider a directed line $\mathcal{L}$ tangent to the circle $C$ at point $A$ and project nodes of $G$ onto the line $\mathcal{L}$. With the same argument as in the proof of Theorem~\ref{th:3-connected}, since the embedding $f$ has the properties in Lemma~\ref{lem:cut} and \ref{lem:embedding}, if we set $V_1$ to be the $|V|/2$ nodes whose projections come first and $V_2$ are the $|V|/2$ nodes whose projections come last, then $G[V_1]$ and $G[V_2]$ are both connected. If $|p(V_1)|\leq 1$, then $(V_1,V_2)$ is a good partition and there is nothing left to prove. Otherwise, we move $\mathcal{L}$ from being tangent at point $A$ to point $B$ ($AB$ is a diameter of the circle $\mathcal{C}$) and consider the resulting partition. Notice that if at point $A$, $p(V_1)>0$, then at point $B$ since $V_1$ and $V_2$ completely switch places compared to the partition at point $A$,  $p(V_1)<0$. Hence, as we move $\mathcal{L}$ from being tangent at point $A$ to point $B$ and keep it tangent to the circle, in the resulting partitions, $p(V_1)$ goes from some positive value to a negative value. Notice that the partition $(V_1,V_2)$ changes only if $\mathcal{L}$ passes a point $D$ on the circle such that at $D$, $\mathcal{L}$ is perpendicular to a line that connects $f(i)$ to $f(j)$ for $i,j\in V$. Now, since the embedding is in general position except for few lines that contain a pseudo-path, if $\{i,j\}\notin E\cup E^*$ or $\{i,j\}\in E\cap E^*$ there are exactly two points on the line that connects $f(i)$ and $f(j)$, so $V_1$ changes at most by one node leaving $V_1$ and one node entering  $V_1$. Thus, in this case $p(V_1)$ changes by at most 2. However, if $\{i,j\}\in E^*$, then since there can be at most $|V|/4-1$ points on the line that connects $f(i)$ to $f(j)$, $p(V_1)$ may change by $|V|/4-1$. So let's consider that $\mathcal{L}$ is perpendicular to the line that connects $f(i)$ and $f(j)$ at point $D$ and  $p(V_1)>0$ slightly before $\mathcal{L}$ passes point $D$ and $p(V_1)\leq 0$ slightly after $\mathcal{L}$ passes point $D$. We consider 2 cases:
    \begin{itemize}
    \item[(a)] Suppose $\{i,j\}\notin E\cup E^*$ or $\{i,j\}\in E\cap E^*$. In this case since $V_1$ changes by at most $2$, therefore, $p(V_1)$ slightly after $\mathcal{L}$ passes $D$ is either $0$ or $-1$. Hence, the partition $(V_1,V_2)$ that we get after $\mathcal{L}$ passes $D$ is a good partition with $|V_1|=|V_2|=|V|/2$.
    \item[(b)] Suppose $\{i,j\}\in E^*$.  Let $(V_{b1},V_{b2})$ be the partition slightly before $\mathcal{L}$ passes $D$ and let $(V_{a1},V_{a2})$ be the partition slightly after. Assume $P_0$ is the pseudo-path between $i$ and $j$ and $P=P_0\cup\{i,j\}$. It is easy to see that $(V_{b1}\backslash P,V_{b2}\backslash P)=(V_{a1}\backslash P,V_{a2}\backslash P)$. 
        If $p(V_{b1}\backslash P)=0$, since $|P_0|<|V|/4$, then $V_1=V_{b1}\backslash P$ and $V_2= V\backslash V_1$ is a good partition and there is nothing left to prove. So either $p(V_{b1}\backslash P)> 0$ or $p(V_{b1}\backslash P)< 0$.

        If $p(V_{b1}\backslash P)> 0$, since $p(V_{a1})\leq 0$ and $V_{b1}\backslash P=V_{a1}\backslash P$, we can add a set of nodes from $P\cap V_{a1}$, say $P'$, to $V_{a1}\backslash P$ to get $p((V_{a1}\backslash P)\cup P')=0$. Now since $|P_0|<|V|/4$, $V_1=(V_{a1}\backslash P)\cup P'$ and $V_2=V\backslash V_1$ is a good partition.

        If $p(V_{b1}\backslash P)< 0$, since $p(V_{b1})> 0$, we can add a set of nodes from $P\cap V_{b1}$, say $P'$, to $V_{b1}\backslash P$ to get $p((V_{b1}\backslash P)\cup P')=0$. Now since $|P_0|<|V|/4$, $V_1=(V_{b1}\backslash P)\cup P'$ and $V_2=V\backslash V_1$ is a good partition.
    \end{itemize}
\end{itemize}
\end{proof}

\medskip

\section{Missing Proofs from Section~\ref{sec:blue_red}}\label{sec:proof3}
\begin{proof}[Proof of Corollary~\ref{cor:blue_red3}]
Suppose without loss of generality that $n_r \geq n_b$ and
let $n_r-n_b=2t$ and $n_r+ n_b =n =2m$.
Set $p(i)=1$ for $i \in R$ and $p(i)=-1$ for $i \in B$.
Then $p(V) =2t$.
From the equations, we have $n_r=m+t$ and $n_b= m-t$.

From Corollary ~\ref{cor:3-general} we can find a partition $(V_1,V_2)$
such that $|V_1|=|V_2|$ and $|p(V_1)-p(V)/2|, |p(V_1)-p(V)/2| \leq 1$.
Let $r_1 = |R \cap V_1|$ and $b_1 = |B \cap V_1|$.
We have $r_1 + b_1 = n/2 =m$ and $t-1 \leq r_1 - b_1 \leq t+1$.
Therefore, $(m+t)/2 - (1/2) \leq r_1 \leq (m+t)/2 +(1/2)$.
Since $r_1$ is an integer and $n_r=m+t$ is even,
it follows that $r_1 = (m+t)/2 = n_r/2$.
Hence, $b_1 = (m-t)/2 = n_b/2$.
Therefore, $V_2$ also contains $n_r/2$ red nodes and
$n_b/2$ blue nodes.
\end{proof}

\bibliographystyle{abbrv}
\bibliography{DBP_bib}
\end{document}